\documentclass[11pt,a4paper]{article}
\usepackage{amsmath}
\usepackage{amsfonts}
\usepackage{amsthm}
\usepackage{amssymb}
\usepackage{url}
\usepackage[usenames,dvips]{color}
\usepackage[dvips]{graphicx}
\usepackage{psfrag, float}  %per al text dels gr�fics

\newtheorem{lemma}{Lemma}[section]
\newtheorem{corollary}[lemma]{Corollary}
\newtheorem{definition}[lemma]{Definition}
\newtheorem{remark}[lemma]{Remark}

\newtheorem{proposition}[lemma]{Proposition}
\newtheorem{theorem}[lemma]{Theorem}

\def\be{\begin{eqnarray}}
\def\ee{\end{eqnarray}}
\def\beal{\begin{aligned}}
\def\enal{\end{aligned}}
\newcommand{\eps}{\varepsilon}
\newcommand{\ga}{\gamma}

\newcommand{\RR}{\mathbb{R}}

\newcommand{\CC}{\mathbb{C}}
\newcommand{\TT}{\mathbb{T}}

\newcommand{\ZZ}{\mathbb{Z}}
\newcommand{\MM}{\mathcal{M}}

\newcommand{\GG}{\mathcal{G}}
\newcommand{\II}{\mathcal{I}}
\newcommand{\BB}{\mathcal{B}}

\newcommand{\KK}{\mathcal{K}}
\newcommand{\KKK}{\mathbb{K}}

\newcommand{\XX}{\mathcal{X}}
\newcommand{\OO}{\mathcal{O}}

\newcommand{\HH}{\mathcal{H}}

\newcommand{\RRR}{\mathcal{R}}

\newcommand{\EE}{\mathcal{E}}
\newcommand{\JJ}{\mathcal{J}}
\newcommand{\AAA}{\mathcal{A}}
\newcommand{\ZZZ}{\mathcal{Z}}
\newcommand{\CCC}{\mathcal{C}}

\newcommand{\Id}{\mathrm{Id}}
\newcommand{\DDD}{\mathcal{D}}

\newcommand{\ii}{^{-1}}
\newcommand{\de}{\delta}
\newcommand{\pa}{\partial}
\newcommand{\la}{\lambda}

\newcommand{\al}{\alpha}
\newcommand{\kk}{\kappa}
\newcommand{\rr}{\rho}

\newcommand{\ol}{\overline}
\newcommand{\La}{\Lambda}
\newcommand{\bet}{\beta}

\newcommand{\wt}{\widetilde}

\begin{document}

\title{Growth of Sobolev norms in the cubic defocusing nonlinear
Schr\"odinger equation with a convolution potential}
\author{Marcel Guardia\thanks{\tt Department of Mathematics, Mathematics Building, University of Maryland,
College Park, MD 20742-4015,  marcel.guardia@upc.edu.}}
\maketitle

\begin{abstract}
Fix $s>1$. Colliander, Keel, Staffilani, Tao and Takaoka proved in
\cite{CollianderKSTT10} the existence of solutions of the cubic defocusing nonlinear
Schr\"odinger equation in the two torus with
$s$-Sobolev norm growing in time. In this paper we generalize their result to the cubic defocusing nonlinear
Schr\"odinger equation with a convolution potential. Moreover, we show that the speed of growth is the same as the one obtained for the cubic defocusing nonlinear Schr\"odinger equation in \cite{GuardiaK12}. The results we obtain can deal with any potential in $H^{s_0}(\TT^2)$, $s_0>0$.
\end{abstract}

\section{Introduction}
The purpose of this paper is to study the growth of Sobolev norms for the periodic cubic defocusing nonlinear Schr\"odinger
equation with a convolution potential,
\begin{equation}\label{def:NLSV}
\left\{\begin{aligned}
&-i \pa_t u+\Delta u+V(x)\ast u=|u|^2 u\\
&u(0,x)=u_0(x),
\end{aligned}\right.
\end{equation}
where $x\in\TT^2=\RR^2/(2\pi\ZZ)^2$, $t\in\RR$,
$u:\RR\times\TT^2\rightarrow\CC$ and $V\in H^{s_0}(\TT^2)$, $s_0>0$, with real Fourier coefficients. This equation is globally well posed in time in any Sobolev space $H^s$ with $s\geq1$. 

If we write the Fourier series 
\[
u(t,x)=\sum_{n\in\ZZ^2}a_n(t) e^{inx}\qquad\text{ and }\qquad V(x)=\sum_{n\in\ZZ^2}v_n e^{inx},
\]
equation \eqref{def:NLSV} becomes an infinite dimensional ordinary differential equation for the Fourier coefficients $a_n$,
\begin{equation}\label{eq:NLSForFourierCoefs}
 -i\dot a_n=\left(|n|^2 +v_n\right)a_n+\sum_{\substack{n_1,n_2,n_3\in \ZZ^2\\n_1-n_2+n_3=n}}
 a_{n_1}\ol{a_{n_2}}a_{n_3}.
\end{equation}
Note that the assumption that $V$ has real Fourier coefficients implies that for this equation $a=0$ is an elliptic critical point.

Equation \eqref{eq:NLSForFourierCoefs} is Hamiltonian since it can be written as
\[
 \dot a_n=2i \pa_{\ \ol{a_n}}\ \HH(a, \ol a),
\]
where
\begin{equation}\label{def:HamForFourier}
\HH (a,\ol a)=\DDD (a,\ol a)+\GG  (a,\ol a)
\end{equation}
with
\begin{align*}
 \DDD (a,\ol a)&=\frac{1}{2}\sum_{n\in\ZZ^2}\left(|n|^2+v_n\right)|a_n|^2\\
\GG  (a,\ol a)&=\frac{1}{4}\sum_{\substack{n_1,n_2,n_3,n_4\in \ZZ^2\\
n_1-n_2+n_3=n_4}}a_{n_1}\ol{a_{n_2}}a_{n_3}\ol{a_{n_4}}.
\end{align*}
Besides $\HH$, equation \eqref{eq:NLSForFourierCoefs}  has another conserved quantity, 
\begin{equation}\label{def:NLS:mass}
 \MM(a)=\sum_{n\in\ZZ^2}|a_n|^2.
\end{equation}
It is usually called mass and  is just the square of the $\ell^2$-norm of the sequence $\{a_n\}_{n\in\ZZ^2}$, which coincides with the $L^2$ norm of $u$.

The problem of growth of $s$-Sobolev norms in Hamiltonian Partial Differential Equations (PDE) has drawn a wide attention in the past decades and was considered by Bourgain one of the next century problems in Hamiltonian PDE \cite{Bourgain00b}. The importance of this phenomenon is that solutions undergoing a large growth of $s$-Sobolev norm with $s>1$ are solutions which, as time evolves, transfer energy to higher and higher modes. The $s$-Sobolev norm is defined by
\[
\|u(t)\|_{H^s(\TT^2)}:=\|u(t,\cdot)\|_{H^s(\TT^2)}:=
\left(\sum_{n\in\ZZ^2} \langle n\rangle ^{2s}|a_n|^2\right)^{1/2},
\]
where $\langle n\rangle=(1+|n|^2)^{1/2}$.
It follows from conservation of energy $\HH(a,\ol a)$ that the $H^1$-norm
of any solution of (\ref{def:NLSV}) is uniformly bounded. Therefore, if  the $H^s$-norm of a solution grows indefinitely for some given $s>1$ while
the $H^1$-norm stays bounded, it is clear that the energy of the solution of \eqref{def:NLSV} must be transferred to higher modes.

In \cite{Bourgain96} (see also \cite{Staffilani}), Bourgain considered the cubic defocusing nonlinear equation 
\begin{equation}\label{def:NLS}
\left\{\begin{aligned}
&-i \pa_t u+\Delta u=|u|^2 u\\
&u(0,x)=u_0(x)
\end{aligned}\right.
\end{equation} 
and obtained upper bounds for the possible growth of Sobolev norms. More concretely, he proved that
\[
\|u(t)\|_{H^s}\le t^{C(s-1)}\|u(0)\|_{H^s} \quad \text{ for }\qquad t\to \infty.
\]
However, he did not obtain orbits obtaining such growth. In \cite{Bourgain00b}, Bourgain posed the following question,
\vskip 0.1in

{\it Are there solutions of \eqref{def:NLS}  
 with periodic boundary conditions in dimension $2$  or higher with unbounded
growth of $H^s$-norm for $s>1$? }

\vskip 0.1in

Moreover, he conjectured, that in case this is true, the upper bound that he had obtained in \cite{Bourgain96} was not optimal and that the growth should
be subpolynomial in time, that is,
\[
\|u(t)\|_{H^s}\ll  t^{\eps}\|u(0)\|_{H^s} \quad \text{ for }
\qquad t\to \infty, \text{ for all }\eps>0.
\]

The question posed by Bourgain is still open. Nevertheless, in the past decades there have been several results in the area which can be classified in two groups. The first group of works obtains upper bounds for the possible growth of $s$-Sobolev norms without showing that this growth  exists. Such results have been obtained for several Hamiltonian Partial Differential Equations (nonlinear Schr\"odinger equation with different nonlinearities, nonlinear wave equation, Hartree equation,...) defined in different manifolds (see \cite{Staffilani97, CollianderDKS01,Bourgain04,Zhong08, CatoireW10, Sohinger11,CollianderKO12,Bourgain96,Sohinger10a,Sohinger10b}).

The second group of results is devoted to the obtention of solutions undergoing growth of $s$-Sobolev norms. Most of the results show arbitrarily large, but finite, growth, as happens in this paper. The first result showing large finite growth is also due to Bourgain. In \cite{Bourgain96}, he
constructs orbits with arbitrarily large growth of  Sobolev norms for the wave equation with a cubic nonlinearity but with a spectrally
defined Laplacian. Similar results are obtained in \cite{GerardG11,Pocovnicu12} for certain nonlinear wave equation. The only results dealing with unbounded growth are  \cite{GerardG10,Pocovnicu11}, which deal with  the Szeg\"o equation, which is integrable, and \cite{Hani11, Hani12}, where unbounded growth  is shown to exist for a pseudo partial differential equation which is a simplification of \eqref{def:NLS}. 
Related to this second group of results, there is also \cite{CarlesF10}, where a spreading of energy among the modes for certain solutions of \eqref{def:NLS} is shown. Nevertheless, this spreading does not lead to growth of Sobolev norms.

Concerning growth of Sobolev norms for the nonlinear Schr\"odinger equation \eqref{def:NLS}, the first result  is due to Kuksin. In \cite{Kuksin97b}, he proves the existence of  solutions with any prescribed growth of $s$-Sobolev norm taking initial data large enough (depending on the prescribed growth). To obtain this result he takes advantage of the fact that for large data the nonlinearity in \eqref{def:NLS} plays a more significant role than the dispersion.

In the present paper we are rather interested in showing growth of Sobolev norms for solutions with small initial data. That is, for solutions  close (in some topology) to the solution $u=0$.  From the dynamical systems point of view, $u=0$ is an elliptic critical point and therefore, showing growth of Sobolev norms for small initial data means that the critical point is unstable in the Sobolev spaces $H^s$, $s>1$. 
The first paper dealing with such setting is the recent paper \cite{CollianderKSTT10}.  In this paper, it is obtained the following result.

\begin{theorem}[\cite{CollianderKSTT10}]\label{thm:Iteam}
Fix $s>1,\ \CCC\gg1$ and $\mu\ll 1$.
Then there exists a global solution $u(t)=u(t,\cdot)$ of \eqref{def:NLS} and $T$ satisfying
 that 
\[
 \|u(0)\|_{H^s}\le \mu,\qquad  \|u(T)\|_{H^s}\ge \CCC.
\]
\end{theorem}
Note that the initial Sobolev norm gives bounds for the mass and energy of the solution $u$, which are constant as time evolves, and therefore are small for all time. 

The paper \cite{CollianderKSTT10} does not give estimates for the time $T$ with respect to the growth of the Sobolev norms, namely estimates of the speed of the growth. These estimates have been obtained recently by the author and Vadim Kaloshin in \cite{GuardiaK12}. This paper contains the following two results. 

\begin{theorem}[\cite{GuardiaK12}]\label{thm:mainGK1}
Let $s>1$. Then, there exists $c>0$ with the following property:
for any large $\KK\gg 1$ there exists a a global solution $u(t)=u(t,\cdot)$
of (\ref{def:NLS}) and a time $T$ satisfying
\[
0 < T \leq \KK^c
\]
such that
% such that for any $t$ with $1< t < T$ we have
%\[
% \|u(t)\|_{H^s} \ge t^{\frac{1}{c}}  \|u(0)\|_{H^s}.
%\]
%In particular,
\[
\|u(T)\|_{H^s}\ge \KK\,\|u(0)\|_{H^s}.
\]
Moreover, this solution can be chosen to satisfy
\[
 \|u(0)\|_{L^2}\le \KK^{-\sigma}.
\]
for some $\sigma>0$ independent of $\KK$.
\end{theorem}

Note that this theorem does not give any information of the initial Sobolev norm but only on its growth. Nevertheless, we want to emphasize that it is dealing with small data since the $L_2$-norm of the solution is very small. One can impose also that the solution has small initial $s$-Sobolev norm at the expense of obtaining a slower growth.

\begin{theorem}[\cite{GuardiaK12}]\label{thm:mainGK2}
Let $s>1$. Then there exists $c>0$ with the following
property: for any small $\mu\ll1$ and large $\CCC\gg 1$ there exists a
a global solution $u(t)=u(t,\cdot)$ of (\ref{def:NLS}) and a time $T$ satisfying
\[
0<T\leq \left(\frac{\CCC}{\mu}\right)^{c\ln (\CCC/\mu)}
\]
such that
\[
\|u(T)\|_{H^s}\ge \CCC\,\,\,\text{ and }\,\,\,\|u(0)\|_{H^s}\leq \mu.
\]
%Moreover, for any $t$ with $1< t < T$, we have
%\[
% \|u(t)\|_{H^s} \ge t^{\frac{1}{c}}  \|u(0)\|_{H^s}.
%\]
\end{theorem}

The purpose of the present paper is to show that the stated results from  \cite{CollianderKSTT10} and \cite{GuardiaK12} are also valid if one adds a convolution potential term to \eqref{def:NLS}. Namely, if one considers equation \eqref{def:NLSV}. The choice of a convolution potential instead of the classical multiplicative potential is, as usual, due to the fact that the term $V(x)\ast u$ is diagonal in the Fourier basis. This fact simplifies considerably the problem. The existence of solutions of \eqref{def:NLSV} with a multiplicative potential instead of a convolution one undergoing growth of Sobolev norms is still not known.

In this paper, we show that the instability mechanism developed in   \cite{CollianderKSTT10,GuardiaK12}  is also valid, with some modifications, for \eqref{def:NLSV} and that therefore for this equation there also exist solutions with arbitrarily high, but finite, growth of Sobolev norm. Moreover, we obtain the same time estimates as the ones given by Theorems \ref{thm:mainGK1} and \ref{thm:mainGK2}.
Note that the cited results  showing growth of Sobolev norms dealt with particular equations. Therefore, it was not clear whether the mechanisms were still valid if one modified slightly the equation. That is,  how robust  the growth of Sobolev norms was. A positive answer to this question  is given in the present paper for the mechanism developed in \cite{CollianderKSTT10} and \cite{GuardiaK12} since it shows that it is still valid when one adds a convolution potential. We want to emphasize that we obtain results for \emph{any potential}. That is, we do not need any smallness condition. We do not need either any non-resonant condition on its Fourier coefficients as usually happens in stability and KAM results (see for instance \cite{BambusiG03, BambusiG06,Bourgain98, GaucklerL10,Eliasson10}). This implies that the reached growth of Sobolev norms is valid for open sets of potentials.

The main results of this paper are the following.
\begin{theorem}\label{thm:main}
Let $s_0>0$ and $s>1$ and take $V\in H^{s_0}(\TT^2)$ with real Fourier coefficients. Then, there exists $c>0$ with the following property:
for any large $\KK\gg 1$ there exists a global solution $u(t)=u(t,\cdot)$
of (\ref{def:NLSV}) and a time $T$ satisfying
\[
0 < T \leq \KK^c
\]
such that
\[
\|u(T)\|_{H^s}\ge \KK\,\|u(0)\|_{H^s}.
\]
Moreover, this solution can be chosen to satisfy
\[
 \|u(0)\|_{L^2}\le \KK^{-\sigma}.
\]
for some $\sigma>0$ independent of $\KK$.
\end{theorem}

\begin{theorem}\label{thm:mainslow}
Let $s_0>0$ and $s>1$ and take $V\in H^{s_0}(\TT^2)$ with real Fourier coefficients. Then, there exists $c>0$ with the following
property: for any small $\mu\ll1$ and large $\CCC\gg 1$ there exists
a global solution $u(t)=u(t,\cdot)$ of (\ref{def:NLSV}) and a time $T$ satisfying
\[
0<T\leq \left(\frac{\CCC}{\mu}\right)^{c\ln (\CCC/\mu)}
\]
such that
\[
\|u(T)\|_{H^s}\ge \CCC\,\,\,\text{ and }\,\,\,\|u(0)\|_{H^s}\leq \mu.
\]
\end{theorem}

\begin{remark}
From the proofs of Theorems \ref{thm:main} and \ref{thm:mainslow}, one can easily see that the initial conditions $u_0$ of $u(t)$ which we obtain in these theorems  do not depend fully on the potential $V$ but only on its norm $\|V\|_{H^{s_0}}$. 
\end{remark}
\begin{remark}
As happened with the results in \cite{CollianderKSTT10} and \cite{GuardiaK12},
even if the just stated theorems are for equation  \eqref{def:NLS} in
the two torus, they can be applied to the $d$ dimensional torus
with $d\geq 2$, since the obtained solutions are also solutions
of equation \eqref{def:NLS} in the $\TT^d$ setting all
the other harmonics to zero. 
\end{remark}

From the dynamical systems point of view, as happened for equation \eqref{def:NLS}, $u=0$ is an elliptic critical point for equation \eqref{def:NLSV}. For equation  \eqref{def:NLS}, this critical point is extremely resonant since its eigenvalues are $\la_n^\pm= \pm i |n|^2$ (see equation \eqref{eq:NLSForFourierCoefs}). It is well known that resonances are a source of instabilities and the mechanism developed in \cite{CollianderKSTT10}  uses them to obtain the growth of Sobolev norms. When one adds the potential and considers equation \eqref{def:NLSV}, these resonances tipically break down. Thus, one would expect that, for well choosen potentials, the growth of Sobolev norms is slower than the one obtained for \eqref{def:NLS} and, therefore,that the critical point $u=0$ is  stable for longer time. Theorems \ref{thm:main} and \ref{thm:mainslow} show that this expected longer stability time does not hold and that the instability phenomena of \eqref{def:NLSV} have the same speed as the ones obtained for
 \eqref{def:NLS}. 
The reason behind this fact is that for equation \eqref{def:NLSV}, since $V\in H^{s_0}$, $s_0>0$, the eigenvalues of $u=0$ satisfy
\[
 \la_n^\pm= \pm i \left(|n|^2+v_n\right)=\pm i |n|^2+\OO\left(\frac{1}{|n|^{s_0}}\right)
\]
and therefore, for high enough modes, the potential contribution to the eigenvalues is negligible. Then, for large enough modes, even if the potential might kill the resonances, one has \emph{almost resonant} terms and one can use them to obtain the growth of Sobolev norms (see Section \ref{sec:SketchProof} for more details).

The instability times obtained in Theorem \ref{thm:mainslow} can be compared with the stability results obtained in  \cite{BambusiG03} (see also \cite{BambusiG06,GaucklerL10}). These papers study the stability of the elliptic critical point $u=0$ of \eqref{def:NLSV} (for certain potentials) in the $H^s$ topology  for $s\geq s_0>1$ and some $s_0$. More concretely, in \cite{BambusiG03} it is obtained the following result. They consider equation \eqref{def:NLSV} with a potential $V$, which satisfies certain non-resonance conditions which is fulfilled by a positive measure set of potentials, and an initial condition $u_0$ satisfying that $\eps=\|u_0\|_{H^s}$ is small enough. Then, they show that the corresponding solution satisfies
\[
 \|u(t)\|_{H^s}\leq 2\eps\,\,\,\text{ for }|t|\leq \frac{C}{\eps^M},
\]
for some constants $C,M>0$ independent of $\eps$. That is, for a positive measure set of potentials, the elliptic critical point $u=0$ of equation \eqref{def:NLSV} is stable for times which are polynomially long with respect to the size of the initial condition in $H^s$. This implies, that the time estimates that we obtain in Theorem \ref{thm:mainslow} are almost optimal, since the instability time cannot be shorter than the time obtained in \cite{BambusiG03}.
% On the contrary, in Theorem \ref{thm:mainslow} we show that for \emph{any potential} there exists solutions whose Sobolev norms grow arbitrarily for times slightly large than polynomial. 

%This result jointly with Theorem \ref{thm:mainslow} gives an almost complete picture of the maximal stability time and the minimal instability time. Indeed, 

Equation \eqref{def:NLSV} was also considered in  \cite{Bourgain98,Eliasson10, FaouG10}. In the first two papers,  the authors prove, the existence of  small quasi-periodic solutions for the nonlinear Schr\"odinger equation \eqref{def:NLSV} in the 2-dimensional torus \cite{Bourgain98}
and in a torus of any dimension \cite{Eliasson10}. These results jointly with  Theorems \ref{thm:main} and \ref{thm:mainslow} show that in any neighborhood of $u=0$ there is coexistence of stable motion (KAM tori) and unstable motion (orbits undergoing growth of $s$-Sobolev norms). In \cite{FaouG10}, the authors
%, assuming a non-resonance condition on $V$,  
prove a Nekhoroshev type result for solutions of \eqref{def:NLSV} with small analytic initial condition. In all these papers, the authors need to assume some non-resonance condition on the potential $V$.

The rest of the paper is devoted to prove Theorems \ref{thm:main} and \ref{thm:mainslow}. These proofs follow the strategy developed to prove Theorem \ref{thm:Iteam} in \cite{CollianderKSTT10}, which was refined in \cite{GuardiaK12} to prove Theorems \ref{thm:mainGK1} and \ref{thm:mainGK2}. 
%They are described in Section \ref{sec:SketchProof}.

\section{Proof of Theorems \ref{thm:main} and \ref{thm:mainslow}}\label{sec:SketchProof}
To prove Theorems \ref{thm:main} and \ref{thm:mainslow} we study equation \eqref{eq:NLSForFourierCoefs} in the Sobolev spaces  $H^s$ with $s> 1$
as well as in the $\ell^1$-space, which is defined as usual as
\[
 \ell^1=\left\{a:\ZZ^2\rightarrow \CC: \|a\|_{\ell^1}=
 \sum_{n\in\ZZ^2}|a_n|<\infty\right\}.
\]
%Let us also point out that the $L^2$-norm conservation of \eqref{def:NLS},
%becomes now conservation of the $\ell^2$-norm of $a$, defined as above.
%Namely, we have that $  \|a(t)\|_{\ell^2}=\|a(0)\|_{\ell^2}$ for all $t\in\RR$.

%We study the evolution of certain solutions of  equation
%\eqref{eq:NLSForFourierCoefs}, which are small in the $\ell^1$ norm.

The first step is to perform some changes of variables that will simplify certain terms of equation \eqref{eq:NLSForFourierCoefs}. These changes will be obtained using the gauge invariance of the equation and performing one step of partial Birkhoff normal form. To take advantage of the gauge invariance we perform the change of coordinates
\begin{equation}\label{def:GaugeChange}
 a_n=e^{iGt}r_n,
\end{equation}
which transforms equation \eqref{eq:NLSForFourierCoefs} into
\[
 -i\dot r_n=\left(|n|^2 +v_n-G\right)r_n+\sum_{\substack{n_1,n_2,n_3\in \ZZ^2\\n_1-n_2+n_3=n}}
 r_{n_1}\ol{r_{n_2}}r_{n_3}.
\]
We choose $G$ properly to remove some terms of the nonlinear part of the equation. We write the sum as
\[
 \sum_{\substack{n_1,n_2,n_3\in \ZZ^2\\n_1-n_2+n_3=n}}=\sum_{n_1,n_3\neq n}+\sum_{n_1=n}+\sum_{n_3=n}-\sum_{n_1=n_3=n}.
\]
The last sum is just one term given by $-|r_n|^2r_n$. The second and third sums can be written as
\[
 r_n\sum_{n_2\in\ZZ^2}|r_{n_2}|^2.
\]
Recall that the $\ell^2$ norm of $a$ is a first integral of equation \eqref{eq:NLSForFourierCoefs} and so the same happens with the $\ell^2$ norm of $r$. Therefore, the second and third sums are equal to $2r_n \|r\|_{\ell_2}$ and choosing $G=2\|r\|_{\ell^2}$ we can remove them. With this choice, we obtain the equation
\begin{equation}\label{eq:NLSAfterGauge}
 -i\dot r_n=\left(|n|^2 +v_n\right)r_n-|r_n|^2r_n+\sum_{\substack{n_1,n_2,n_3\in \ZZ^2\\n_1-n_2+n_3=n,n_1,n_3\neq n}}
 r_{n_1}\ol{r_{n_2}}r_{n_3}.
\end{equation}
It  is a Hamiltonian system with respect to 
\begin{equation}\label{def:HamAfterGauge}
\HH' (r,\ol r)=\DDD (r,\ol r)+\GG'  (r,\ol r),
\end{equation}
where $\DDD$ has been introduced in \eqref{def:HamForFourier} and 
\[
\GG' (r,\ol r)=-\frac{1}{4}|r_n|^4+\frac{1}{4}\sum_{\substack{n_1,n_2,n_3,n_4\in \ZZ^2\\
n_1-n_2+n_3=n_4, n_1,n_3\neq n_4}}r_{n_1}\ol{r_{n_2}}r_{n_3}\ol{r_{n_4}}.
\]
Next step is to perform one step of partial Birkhoff normal form to reduce the size of certain terms in equation \eqref{eq:NLSAfterGauge}. In the normal form procedure we need to treat in a considerably different way the low and high harmonics. The reason is the following. The eigenvalues of the critical point $r=0$ of equation \eqref{eq:NLSAfterGauge} are given by
\[
 \la_n=\pm i(|n|^2+v_n).
\]
Recall that we assume that $v_n\in\RR$ and therefore the eigenvalues are purely imaginary. Recall also that by hypothesis $V\in H^{s_0}(\TT^2)$ and thus its Fourier coefficients decay polynomially in $|n|$. This implies that for low harmonics the term $v_n$ makes a big influence in the eigenvalues whereas for very high harmonics, the eigenvalues $\la_n$ are extremely close to $\pm i|n|^2$. To quantify this fact, we define a constant $\kk_0$ such that 
\begin{equation}\label{def:kappa0}
 |v_n|\leq \frac{1}{100}\,\,\,\text{ for }|n|\geq \kk_0.
\end{equation}
Throughout the paper we call low modes  to the modes satisfying $|n|<\kk_0$ and high modes to the modes satisfying $|n|\geq \kk_0$. We also use the notation
\[
 B(\kk_0)=\left\{n\in\ZZ^2: |n|<\kk_0\right\}.
\]
We  perform a partial normal form procedure treating differently each monomial depending whether $|n|\leq \kk_0$ or $|n|\geq \kk_0$. Before stating the theorem, we introduce the coefficients  $\rr_{n_1n_2n_3n_4}$, which are the small divisors that will arise in the normal form,
\begin{equation}\label{def:SmallDivisor}
 \rr_{n_1n_2n_3n_4}=\sum_{j=1}^4(-1)^{j+1}\left(|n_j|^2+v_{n_j}\right).
\end{equation}

\begin{theorem}\label{thm:NormalForm}
Fix $\eta\in (1/4,1/3)$. There exists a symplectic
change of coordinates $r=\Gamma(\alpha)$ in a neighborhood of $0$ in $\ell^1$
which transforms the Hamiltonian $\HH'$ in \eqref{def:HamAfterGauge} into the Hamiltonian,
\[
 \HH'\circ\Gamma=\DDD+\wt \GG_1+\wt \GG_2+\RRR,
\]
where $\wt \GG_1$ is defined as 
\[
 \wt \GG_1(\al, \ol\al)=-\frac{1}{4}\sum_{n\in B(\kk_0)}|\al_n|^4+\frac{1}{4}\sum_{\substack{(n_1,n_2,n_3,n_4)\in
\II\\n_1-n_2+n_3=n_4, n_1,n_3\neq n_4\\\left|\rr_{n_1n_2n_3n_4}\right|\leq \eta}}\al_{n_1}\ol{\al_{n_2}}
\al_{n_3}\ol{\al_{n_4}},
\]
with
\[
 \II=\left\{(n_1,n_2,n_3,n_4): n_i\in\ZZ^2\text{ and at least one of the }n_i\text{'s satisfies }n_i\in B(\kk_0)\right\},
\]
$\wt \GG_2$ is defined as 
\[
\wt \GG_2(\al, \ol\al)=-\frac{1}{4}\sum_{n\in\ZZ^2\setminus B(\kk_0)}|\al_n|^4+\frac{1}{4}\sum_{\substack{n_i\in
\ZZ^2\setminus B(\kk_0), i=1,\ldots,4\\n_1-n_2+n_3=n_4,n_1,n_3\neq n_4 \\ \sum_{i=1}^4(-1)^{i}|n_i|^2=0}}\al_{n_1}\ol{\al_{n_2}}
\al_{n_3}\ol{\al_{n_4}}
\]
and $X_\RRR$, the vector field associated to
the Hamiltonian $\RRR$, satisfies
\[
 \|\XX_\RRR\|_{\ell^1}\leq \OO\left(\|\al\|_{\ell^1}^5\right).
\]
Moreover, the change $\Gamma$ satisfies
\[
 \left\|\Gamma-\mathrm{Id}\right\|_{\ell^1}\leq
\OO\left(\|\al\|_{\ell^1}^3\right).
\]
\end{theorem}
The proof of this theorem is postponed to Section \ref{app:NormalForm}.

Once we have used the gauge invariance and performed one step of partial Birkhoff normal form, we have a new vector field
\begin{equation}\label{eq:InfiniteODEAfterNormalForm}
\begin{split}
- i \dot \alpha_n =& \left(|n|^2 +v_n\right)\alpha_n-|\alpha_n|^2\alpha_n +
\sum_{(n_1,n_2,n_3)\in \AAA_0 (n)} \al_{n_1} \overline { \al_{n_2}}\al_{n_3}\\
&+\sum_{(n_1,n_2,n_3)\in \AAA_1 (n)} \al_{n_1} \overline { \al_{n_2}}\al_{n_3}+
2\partial_{\overline \al_n} \RRR,
\end{split}
\end{equation}
where the sets  $\AAA_0 (n)$ and $ \AAA_1 (n)$ are introduced in the next definition. 
\begin{definition}\label{definition:ResonantSets}
We define the sets $\AAA_0 (n)$ and $ \AAA_1 (n)$ as follows. 
\begin{itemize} \item For $n\in B(\kk_0)$
\begin{itemize}
\item[(i)] $(n_1,n_2,n_3)\in \AAA_0 (n)$ provided $n_1-n_2+n_3=n$, $n_1\neq n$, $n_3\neq n$,  $\left|\rr_{n_1n_2n_3n_4}\right|\leq \eta$
\item[(ii)] $\AAA_1(n)=\emptyset$.
\end{itemize}
 \item For $n\in \ZZ^2\setminus B(\kk_0)$
\begin{itemize}
 \item[(i)] $(n_1,n_2,n_3)\in \AAA_0 (n)$ provided $n_1-n_2+n_3=n$, $n_1\neq n$, $n_3\neq n$, $\left|\rr_{n_1n_2n_3n_4}\right|\leq \eta$ and at least one of the $n_i$ satisfies $n_i\in B(\kk_0)$.
 \item[(ii)] $(n_1,n_2,n_3)\in \AAA_1 (n)$ provided $n_1-n_2+n_3=n$,  $n_1\neq n$, $n_3\neq n$, $n_j\in\ZZ^2\setminus B(\kk_0)$, $j=1,2,3$ and $\sum_{j=1}^4(-1)^{j}|n_j|^2\neq 0$.
\end{itemize}
\end{itemize}
%\begin{equation}\label{def:ResonantLatticeBeforeGauge}
%\begin{split}
% \AAA_0(n)=\Big\{&(n_1,n_2,n_3)\in\left(\ZZ^2\right)^3:\,n_1-n_2+n_3=n,\\
%&|n_1|^2-|n_2|^2+|n_3|^2=|n|^2\Big\}.
%\end{split}
%\end{equation}
\end{definition}
To remove the linear part of this vector field, we perform a change to  rotating coordinates
\begin{equation}\label{def:rotatingcoord}
 \al_n=e^{i(|n|^2+v_n)t}\beta_n.
\end{equation}
Then, we obtain the vector field
\begin{equation}\label{eq:InfiniteODEARotating}
\begin{split}
- i \dot \beta_n =&-|\bet_n|^2\bet_n+ \sum_{(n_1,n_2,n_3)\in \AAA_0 (n)} \bet_{n_1} \overline { \bet_{n_2}}\bet_{n_3}e^{i\rr_{n_1n_2n_3n}t}\\
&+ \sum_{(n_1,n_2,n_3)\in\AAA_1(n)} \bet_{n_1} \overline { \bet_{n_2}}\bet_{n_3}e^{i\rr_{n_1n_2n_3n}t}+ \RRR_n'(\beta),
\end{split}
\end{equation}
where $\rr_{n_1n_2n_3n}$ has been defined in \eqref{def:SmallDivisor} and $\RRR'$ is the vector field associated to the Hamiltonian $\RRR(\{e^{i(|n|^2+v_n)t}\beta_n\})$.

We want to consider a ``good'' first order of this vector field. To this end we take into acount that for (resonant) nonlinear terms involving only high harmonics, thanks to the decay of the Fourier coefficients of $V$, one can choose the coefficients $\rr_{n_1n_2n_3n}$ to be very small. Thus, we rewrite this vector field as
\begin{equation}\label{eq:VFFinal}
\begin{split}
- i \dot \beta_n =& -|\bet_n|^2\bet_n+ \sum_{(n_1,n_2,n_3)\in \AAA_0 (n)} \bet_{n_1} \overline { \bet_{n_2}}\bet_{n_3}e^{i\rr_{n_1n_2n_3n}t}\\
&+ \sum_{(n_1,n_2,n_3)\in \AAA_1 (n)} \bet_{n_1} \overline { \bet_{n_2}}\bet_{n_3}+ \wt \RRR_n(\beta),
\end{split}
\end{equation}
where
\begin{equation}\label{def:FinalRemainder}
  \wt \RRR= \RRR'+\JJ,
\end{equation}
with
\begin{equation}\label{def:RemainderRotating}
\JJ_n(\bet)=\sum_{(n_1,n_2,n_3)\in \AAA_1 (n)} \bet_{n_1} \overline { \bet_{n_2}}\bet_{n_3}\left(e^{i\rr_{n_1n_2n_3n}t}-1\right).
\end{equation}
To obtain orbits undergoing growth of Sobolev norms for this vector field, we first consider the truncated vector field
\begin{equation}\label{eq:VFFinal:Truncated}
- i \dot \beta_n =-|\bet_n|^2\bet_n+  \sum_{(n_1,n_2,n_3)\in \AAA_0 (n)} \bet_{n_1} \overline { \bet_{n_2}}\bet_{n_3}e^{i\rr_{n_1n_2n_3n}t}+ \sum_{(n_1,n_2,n_3)\in \AAA_1 (n)} \bet_{n_1} \overline { \bet_{n_2}}\bet_{n_3}.
\end{equation}
We will study certain orbits of this vector field with growth of Sobolev norms and later on we will show that close to them are orbits of \eqref{eq:VFFinal} undergoing the same growth. To obtain such orbits for \eqref{eq:VFFinal:Truncated} we will use the instability mechanism developed in \cite{CollianderKSTT10}. In that paper, which as we have explained deals with equation \eqref{def:NLS}, the authors select a very cleverly chosen finite set of modes. The modes in this set on the one hand do not interact with the modes out of this set and on the other hand interact between themselves in a very particular way that leads to transfer of energy. To use this mechanism, we need to take into account that in equation \eqref{eq:VFFinal:Truncated} there are two types of interaction between modes. The first one is through tuples in
\[
\wt\AAA_0=\Big\{(n_1,n_2,n_3,n_4)\in\left(\ZZ^2\right)^4:\,(n_1,n_2,
n_3)\in\AAA_0(n_4)\Big\},
\]
which contain at least one low mode (see Definition \ref{definition:ResonantSets}). The second one is through tuples in 
\[
\wt\AAA_1=\Big\{(n_1,n_2,n_3,n_4)\in\left(\ZZ^2\right)^4:\,(n_1,n_2,
n_3)\in\AAA_0(n_1)\Big\},
\] 
which only contain high modes. As we have already mentioned, the geometry of the resonances differ considerably depending whether $(n_1,n_2,n_3,n_4)\in\wt \AAA_0$ or $(n_1,n_2,n_3,n_4)\in\wt \AAA_1$. In the second case, since we are omitting the influence of the convolution potential, the four points  form a
rectangle in $\ZZ^2$. These were the resonances studied in \cite{CollianderKSTT10} (and also in \cite{GuardiaK12}) since in these papers they consider equation \eqref{def:NLS} and therefore $v_n=0$. However, now the geometry of  tuples involving low harmonics  changes considerably.  To overcome this problem we take advantage of the fact that  the mechanism developed in   \cite{CollianderKSTT10} deals with orbits supported only in high harmonics and therefore they can be choosen to satisfy $|n|\geq \kk_0$, where $\kk_0$ is the constant in \eqref{def:kappa0}. Nevertheless, we  need to further modify the support to avoid interactions between high and low harmonics.

We reduce \eqref{eq:VFFinal:Truncated} to
a finite-dimensional system, which corresponds to an invariant
finite-dimensional plane.  We explain now the construction of this finite set of modes $\Lambda\subset\ZZ^2\setminus B(\kk_0)$, where $\kk_0$ is the constant defined in \eqref{def:kappa0}, and impose
additional conditions on $\Lambda$ from  the ones considered in \cite{CollianderKSTT10}. Fix $N\gg 1$.
Following \cite{CollianderKSTT10} we define a set
$\La\subset \ZZ^2\setminus B(\kk_0)$ consisting of $N$ pairwise disjoint \emph{generations}:
\[
 \Lambda=\Lambda_1\cup\ldots\cup\Lambda_N.
\]
Define a \emph{nuclear family} to be a rectangle $(n_1,n_2,n_3,n_4)\in\wt\AAA_1$,
such that $n_1$ and $n_3$ (known as the \emph{parents}) belong to a generation
$\Lambda_j$ and $n_2$ and $n_4$ (known as the \emph{children}) live in the next
generation $\Lambda_{j+1}$. Note that if $(n_1,n_2,n_3,n_4)$ is a nuclear
family, then so are  $(n_1,n_4,n_3,n_2)$,  $(n_3,n_2,n_1,n_4)$ and
$(n_3,n_4,n_1,n_2)$. These families are called trivial permutations of the
family $(n_1,n_2,n_3,n_4)$.

The conditions imposed to the set $\La$ in \cite{CollianderKSTT10}
 are
\begin{description}
\item[$1_\Lambda$\ {\it Closure}]\
If $n_1, n_2, n_3\in\Lambda$ and $(n_1,n_2,n_3)\in\wt \AAA_1(n)$,
then $n\in\Lambda$. That is, if three vertices of a rectangle are in $\La$ so is
    the last fourth one.
\item[$2_\Lambda$\ {\it Existence and uniqueness of spouse and children}]\
For any $1\leq j <N$ and any $n_1\in\Lambda_j$, there exists a unique nuclear
family $(n_1,n_2,n_3,n_4)$ (up to trivial permutations) such that $n_1$ is
a parent of this family. In particular, each $n_1\in\Lambda_j$ has
a unique spouse $n_3\in\Lambda_j$ and has two unique children
$n_2,n_4\in\Lambda_{j+1}$ (up to permutation).

\item[$3_\Lambda$\ {\it Existence and uniqueness of sibling and parents}]\
For any $1\leq j <N$ and any $n_2\in\Lambda_{j+1}$, there exists a unique
nuclear family $(n_1,n_2,n_3,n_4)$ (up to trivial permutations) such that
$n_2$ is a child of this family. In particular each $n_2\in\Lambda_{j+1}$
has a unique sibling $n_4\in\Lambda_{j+1}$ and two unique parents
$n_1,n_3\in\Lambda_{j}$ (up to permutation).

\item[$4_\Lambda$\ {\it Nondegeneracy}]\
The sibling of a frequency $n$ is never equal to its spouse.

\item[$5_\Lambda$\ {\it Faithfulness}]\
Apart from the nuclear families, $\Lambda$ does not contain any other rectangle.
\end{description}
To slow down the spreading of mass to high modes out of $\Lambda$. We impose the following condition, which was already considered in \cite{GuardiaK12} for the resonant sets of \eqref{def:NLS}.
\begin{description}
\item[$6_\Lambda$\ {\it No spreading to high harmonics}]\
Let us consider $n\in \ZZ^2\setminus \Lambda$. Then, there exists at most two tuples $(n_1,n_2,n_3,n_4)\in\wt\AAA_1$ (up to permutation) containing $n$  and two modes belonging to $\Lambda$.
\end{description}
Now we  have to impose the following conditions to avoid interactions between the modes of $\Lambda$ and the low harmonics in $B(\kk_0)$, which possess a completely different geometry of resonances.
\begin{description}
\item[$7_\Lambda$\ {\it Absence of low harmonics resonant interactions}]\
Take $n_1,n_2,n_3\in\Lambda$. Then, for any $n\in B(\kk_0)\subset\ZZ^2$, one has that $(n_1,n_2,n_3,n)\not\in \wt\AAA_0$ (nor any permutation of it).
\item[$8_\Lambda$\ {\it No spreading to low harmonics}]\
Let us consider $n\in \ZZ^2\setminus \Lambda$. Then, there exist no  tuples $(n_1,n_2,n_3,n)\in \wt \AAA_0$ containing $n$ and two modes belonging to $\Lambda$ and the same happens for any permutation of the tuple.
\end{description}

\begin{theorem}\label{thm:SetLambda}
Let $\KK\gg 1$. Then, there exists
$N\gg 1$ large and a set $\Lambda\subset\ZZ^2$, with
\[
 \Lambda=\Lambda_1\cup\ldots\cup\Lambda_N,
\]
which satisfies  conditions $1_\Lambda$ -- $8_\Lambda$ and also
\begin{equation}\label{def:Growth}
 \frac{\sum_{n\in\Lambda_{N-1}}|n|^{2s}}{\sum_{n\in\Lambda_3}|n|^{2s}}\geq
 \dfrac 12 2^{(s-1)(N-4)}\ge \KK^2.
\end{equation}
Moreover,  we can ensure that each generation $\Lambda_j$ has $2^{N-1}$ disjoint frequencies $n$ satisfying
\[
e^{\eta_1 N^2}\leq |n|\leq e^{\eta_2 N^2},
\]
for certain constants $\eta_1,\eta_2>0$ independent of $\KK$.
\end{theorem}
The construction of such set $\Lambda$  was first done in 
 \cite{CollianderKSTT10} for \eqref{def:NLS} (see Proposition 2.1 of that paper). In \cite{GuardiaK12} the construction was refined adding condition $6_\Lambda$ but still for equation \eqref{def:NLS}. In Section \ref{app:SmallSobolev} we show how to adapt the construction of $\Lambda$ to the present setting and we prove Theorem \ref{thm:SetLambda}.

As explained in  \cite{CollianderKSTT10}, thanks to  Properties $1_\Lambda$ and $7_\La$ satisfied by the set $\Lambda$,  the manifold
\[
\MM=\left\{\beta\in\CC^{\ZZ^2}: \beta_n=0\,\,\text{for all }
n\not\in\Lambda\right\}
\]
is invariant by the flow associated to \eqref{eq:VFFinal:Truncated}. This manifold has  finite dimension which  is  equal to $|\Lambda|=N2^{N-1}$.    Equation  \eqref{eq:VFFinal:Truncated}  restricted to $\MM$
reads
\begin{equation}\label{eq:InftyODE:FirstFiniteReduction}
 -i\dot\bet_n=-\bet_n|\bet_n|^2+2
\beta_{n_{\mathrm{child}_1}}\beta_{n_{\mathrm{child}_2}}\ol{\beta_{n_\mathrm{spouse}}}+2
\beta_{n_{\mathrm{parent}_1}}\beta_{n_{\mathrm{parent}_2}}\ol{\beta_{n_\mathrm{sibling}}}.
\end{equation}
for each $n\in \La$. The presence of parents, children, and the sibling are guaranteed by $2_\La$
and $3_\La$. In the first and last generations, the parents and
children are set to zero respectively. Note that this system coincides with  system (2.14) in \cite{CollianderKSTT10} and (16) in  \cite{GuardiaK12}.  
%The reason is that, as we have already explained,  we are only dealing with very high modes and thus we are neglecting, as first order, the terms coming from the convolution potential. 
 To analyze this first order supported on the set $\Lambda$, we just  recall the results obtained in \cite{CollianderKSTT10, GuardiaK12}. The first step is to take into account that the manifold $\MM$ has a submanifold of
considerably lower dimension which is also invariant.

\begin{corollary}[\cite{CollianderKSTT10}]\label{coro:Invariant}
Consider the subspace
\[
 \wt \MM=\left\{\beta\in\MM: \beta_{n_1}=\beta_{n_2}\,\,\text{for all
}n_1,n_2\in\Lambda_j\,\,\text{for some }j\right\},
\]
where all the members of a generation take the same value. Then, $\wt \MM$ is invariant under the flow associated to \eqref{eq:InftyODE:FirstFiniteReduction}.
\end{corollary}

The dimension of  $\wt \MM$ is equal to the number of generations, namely $N$.
To define equation
\eqref{eq:InftyODE:FirstFiniteReduction} restricted to $\wt \MM$, we consider
\begin{equation}\label{def:ChangeToToyModel}
 b_j=\beta_n\,\,\,\text{ for any }n\in\Lambda_j.
\end{equation}
Then \eqref{eq:InftyODE:FirstFiniteReduction}  restricted to $\wt \MM$ becomes
\begin{equation}\label{def:model}
 \dot  b_j=-ib_j^2\ol b_j+2i \ol
b_j\left(b_{j-1}^2+b_{j+1}^2\right),\,\,j=0,\ldots N.
\end{equation}
This model was first derived in \cite{CollianderKSTT10} and was studied in great detail in that paper and in \cite{GuardiaK12}.
% It is a Hamiltonian system with  respect to the Hamiltonian
%\[
% h(b):=\frac{1}{4}\sum_j |b_j|^4-
% \frac{1}{2}\sum_j \left(\ol b_j^2b_{j-1}^2+b_j^2\ol b_{j-1}^2\right)
%\]
%and the  symplectic form $\Omega=\frac{i}{2}db_j\wedge d\ol b_j$.

\begin{theorem}[\cite{GuardiaK12}]\label{thm:ToyModelOrbit}
Fix a large $\gamma\gg 1$.  Then for any large enough $N$ and
$\de=e^{-\ga N}$, there exists an orbit of system \eqref{def:model},
$\nu>0$ and $T_0>0$ such that
\[
\begin{split}
 |b_3(0)|&>1-\de^\nu\\
|b_j(0)|&< \de^\nu\qquad\text{ for }j\neq 3
\end{split}
\qquad \text{ and }\qquad
\begin{split}
 |b_{N-1}(T_0)|&>1-\de^\nu\\
|b_j(T_0)|&<\de^\nu \qquad\text{ for }j\neq N-1.
\end{split}
\]
Moreover, there exists a constant $\KKK>0$
independent of  $N$ such that $T_0$ satisfies
\begin{equation}\label{def:Time:ToyModel}
 0<T_0<\KKK N \ln \left(\frac 1 \de \right)=\KKK\,\ga\,N^2.
\end{equation}
\end{theorem}

This theorem is proved in \cite{GuardiaK12} using dynamical systems tools such as normal forms and the Shilnikov boundary problem.
% The use of these tools allow to obtain a fast speed of growth of Sobolev norms.  

%\begin{remark} This  proposition is also valid if one considers $\de=C^{-\KK^2}$ for any fixed $C>0$. This will allow us to prove Theorem \ref{thm:mainslow}
%\end{remark}

From the orbit obtained in Theorem \ref{thm:ToyModelOrbit} and using the change
  \eqref{def:ChangeToToyModel} one can obtain an orbit of equation \eqref{eq:VFFinal:Truncated}. Moreover, both equations \eqref{eq:VFFinal:Truncated} and
\eqref{def:model} are invariant under  rescaling
\begin{equation}\label{def:Rescaling}
b^\lambda(t)=\lambda^{-1}b\left(\lambda^{-2}t\right).
\end{equation}
Then, by Theorem \ref{thm:ToyModelOrbit}, the time spent by the solution $b^\lambda(t)$ is
\begin{equation}\label{def:Time:Rescaled}
 T=\la^2 T_0\le \la^2\,\KKK\,\ga\,N^2,
\end{equation}
where $T_0$ is the time obtained in Theorem \ref{thm:ToyModelOrbit}.

We see now that there is a solution
of equation \eqref{eq:VFFinal}
which is close to the orbit $\beta^\la$ of
\eqref{eq:VFFinal:Truncated} defined as
\begin{equation}\label{def:RescaledApproxOrbit}
\begin{split}
\beta^\la_n(t)&=\lambda^{-1}b_j\left(\lambda^{-2}t\right)
\,\,\,\text{ for each }\ n\in\Lambda_j\\
\beta^\la_n(t)&=0\,\,\,\text{ for each }\ n\not\in\Lambda,
\end{split}
\end{equation}
where $b(t)$ is the orbit given by Theorem \ref{thm:ToyModelOrbit}. To obtain this approximation result we need to consider a large $\lambda$. The reason, as it is explained in \cite{CollianderKSTT10, GuardiaK12}, is that in order the original system to be well
approximated by the truncated one, we need the cubic terms  in
\eqref{eq:VFFinal} to dominate over
the quintic ones that are in the remainder $\wt \RRR$.  Nevertheless, due to \eqref{def:Time:Rescaled}, the bigger $\la$, the slower
the instability
time. The choice of  $\la$ (with respect to $N$) that we have to do is different for the proof of Theorems \ref{thm:main} and \ref{thm:mainslow}. Both choices are the same as the ones done in \cite{GuardiaK12}. We show here how to prove Theorem \ref{thm:main} and at the end of the section we will show how to adapt the proof to deal with Theorem \ref{thm:mainslow}.

The approximation result needed to prove Theorem \ref{thm:main} is the following.
\begin{theorem}\label{thm:Approximation}
Let  $\bet^\la(t)=\{\bet^\la_n(t)\}_{n\in \ZZ^2}$ be the solution of
\eqref{eq:VFFinal:Truncated} given
by \eqref{def:RescaledApproxOrbit}
and $T$  the time defined in \eqref{def:Time:Rescaled}. Consider $\wt \bet(t)=\{\wt\bet_n(t)\}_{n\in \ZZ^2}$
 the solution of \eqref{eq:VFFinal} with the same initial condition as $\bet^\la$, that is $\wt \bet(0)=\bet^\la(0)$.
Then, there exist a constant $\kk>0$ independent of $N$ and $\ga$ such that, for
\begin{equation}\label{def:LambdaOfN}
 \la=e^{\kk\ga N},
\end{equation}
and $0<t<T$ we have
\begin{equation}\label{eq:Approx:BoundDiff}
\sum_{n\in \ZZ^2}
\left|\wt \bet_n(t)-\beta^\la_n(t)\right|
 \le \la^{-2}.
\end{equation}
\end{theorem}
The proof of this theorem is deferred to Section \ref{app:Approx}. Now we are ready to prove Theorem \ref{thm:main}. We proceed as in \cite{GuardiaK12}.

%\begin{theorem}\label{thm:Approximation}
%Let $\alpha(t)=\{\alpha_n(t)\}_{n\in \ZZ^2}$ be a solution to
%\[
%- i \dot \alpha_n = |n|^2 \alpha_n +
%\sum_{(n_1,n_2,n_3)\in \AAA_0 (n)} \al_{n_1} \overline { \al_{n_2}}\al_{n_3}+
%\partial_{\overline \al_n} \RRR',
%\]
%where $\RRR'=O(\|\al(0)\|^6_{\ell^1})$.
%Let $\bet(t)=\{\bet_n(t)\}_{n\in \ZZ^2}$ be a solution to
%(\ref{eq:InftyODE:AfterVariation}).
%Suppose supp $\alpha(0)\subset \La$, $\alpha(0)=\bet(0)$,
%and $\la=\|\al(0)\|_{\ell^\infty}^{-1}$. Then for $T$ satisfying
%$Te^{-\la^2 T}\la^6<1$ and $0<t<T$ we have
%\[
%\sum_{n\in \ZZ^2}
%|\alpha_n(t)-e^{i(G+|n|^2)t}\beta_n(t)|
% \le \la^{-2}.
%\]
%\end{theorem}

%From the results stated in the three key theorems, that is, Theorems \ref{thm:NormalForm}, %\ref{thm:ToyModelOrbit} and \ref{thm:Approximation},
%we are ready to finish the  proof Theorem \ref{thm:main}.

\begin{proof}[Proof of Theorem \ref{thm:main}]
Using the change of variables \eqref{def:rotatingcoord} and the change $\Gamma$ obtained in Theorem \ref{thm:NormalForm}, from the solution $\wt\bet(t)$  obtained in Theorem \ref{thm:Approximation} we define  
\[
r(t)=\Gamma \left(\left\{\wt\bet_n(t) e^{i(|n|^2+v_n)t}\right\}\right),
\]
which is a solution of system \eqref{eq:NLSAfterGauge}. Note that the change \eqref{def:GaugeChange}, which relates $r$ and the original variable $a$ preserves the Sobolev norms and therefore, to prove Theorem \ref{thm:main}, one only needs to check  the stated properties on   $r(t)$  instead of on $a(t)$.

To compute the growth of Sobolev norm of the orbit $r(t)$, we define
\begin{equation}\label{def:Sums}
 S_j=\sum_{n\in\Lambda_j}|n|^{2s}  \text{ for }j=1,\dots, N-1.
\end{equation}
%Recall also that
%$2^{N-1}=\sum_{n\in\Lambda_j}1=|\Lambda_j|$. 
We want to prove that
\[
\frac{\|r(T)\|_{H^s}}{\|r(0)\|_{H^s}}\gtrsim \KK
\]
and estimate the mass $\|r(0)\|_{L^2}$ of the solution. Since we want to use the estimate obtained in Theorem \ref{thm:SetLambda}, we start by bounding $\|r(T)\|_{H^s}$ in terms of
$S_{N-1}$ as
\[
\left \|r(T)\right\|^2_{H^s}\geq \sum_{n\in\Lambda_{N-1}}|n|^{2s} \left|r_n(T)\right|^2\geq S_{N-1}\inf_{n\in\Lambda_{N-1}}\left|r_n(T)\right|^2,
\]
Now it is enough to obtain a lower bound for $\left|r_n(T)\right|$ for $n\in\Lambda_{N-1}$. Using the results of Theorems \ref{thm:NormalForm} and \ref{thm:Approximation}, we obtain
\begin{equation}\label{eq:ComputationFinalSobolev}
\begin{split}
 \left|r_n(T)\right|\geq & \left|\wt\bet_n(T)\right|-\left|\Gamma_n \left(\left\{\wt\bet_n(T) e^{i(|n|^2+v_n)T}\right\}\right)(T)-\wt\bet_n(T) e^{i(|n|^2+v_n)T}\right|\\
\geq &\left|\beta^\la_n(T) \right|- \left|\wt\bet_n(T)-\beta^\la_n(T) \right|\\
&-\left|\Gamma_n \left(\left\{\wt\bet_n(T) e^{i(|n|^2+v_n)T}\right\}\right)(T)-\wt\bet_n(T) e^{i(|n|^2+v_n)T}\right|.
\end{split}
\end{equation}
We need to obtain a lower bound for the first term of the right hand side and upper bounds for the second and third ones. Using the definition of $\beta^\la$ in \eqref{def:RescaledApproxOrbit}, the relation between $T$ and $T_0$ established in \eqref{def:Time:Rescaled} and the results in Theorem \ref{thm:ToyModelOrbit}, we have that for $n\in\Lambda_{N-1}$,
\[
\left|\beta_n^\la(T) \right|=   \la^{-1}\left| b_{N-1}(T_0)\right|^2 \geq \frac{3}{4}\la^{-1}.
\]
For the second term in the right hand side of \eqref{eq:ComputationFinalSobolev}, it is enough to use Theorem \ref{thm:Approximation} to obtain,
\[
\left|\wt\bet_n(T)-\beta^\la_n(T) \right|\leq \left(\sum_{n\in\ZZ^2}\left|\wt\bet_n(T)-\beta^\la_n(T) \right|\right) \leq \frac{\la^{-1}}{8}.
\]
For the  lower bound of the third term, we use the bound for $\Gamma-\mathrm{Id}$ given in Theorem \ref{thm:NormalForm}. Then,
\[
\begin{split}
\Big|\Gamma_n &\left(\left\{\wt\bet_n(T) e^{i(|n|^2+v_n)T}\right\}\right)(T)-\wt\bet_n(T) e^{i(|n|^2+v_n)T}\Big|\\ 
&\leq  \left\|\Gamma_n \left(\left\{\wt\bet_n(T) e^{i(|n|^2+v_n)T}\right\}\right)(T)-\wt\bet_n(T) e^{i(|n|^2+v_n)T}\right\|_{\ell^1}
\leq  \frac{\la^{-1}}{8}.
\end{split}
\]
Thus, we can conclude that
\begin{equation}\label{eq:GrowthSobolev:Final}
 \left \|r(T)\right\|^2_{H^s}\geq \frac{\la^{-2}}{4}S_{N-1}.
\end{equation}
Now we prove that
\begin{equation}\label{eq:GrowthSobolev:Initial}
 \left \|r(0)\right\|^2_{H^s}\lesssim \la^{-2}S_3.
\end{equation}
To prove it, we start by using Theorem \ref{thm:NormalForm} to see that
\[
\left\|r(0)\right\|^2_{H^s}\leq \sum_{n\in\ZZ^2}|n|^{2s} \left|\wt\bet_n(0)+\left(\Gamma_n\left(\wt\bet(0)\right)-\wt\bet_n(0)\right)\right|^2.
\]
We first bound $\left\|\wt \bet(0)\right\|^2_{H^s}$. To this end, let us recall that $\mathrm{supp }\ \wt\bet=\Lambda$ and that $\wt\bet (0)=\bet^\la (0)$ (see Theorem \ref{thm:Approximation}). Therefore,
\[
\left\|\wt \bet(0)\right\|^2_{H^s}=\sum_{n\in\Lambda}|n|^{2s} \left|\wt \bet_n(0)\right|^2=\sum_{n\in\Lambda}|n|^{2s} \left|\bet^\la_n(0)\right|^2.
\]
Then, recalling the definition of $\beta^\la$ in \eqref{def:RescaledApproxOrbit} and the results in Theorem \ref{thm:ToyModelOrbit},
\[
\begin{split}
\sum_{n\in\Lambda}|n|^{2s}\left|\beta^\la_n(0)\right|^2&\leq \la^{-2} \left(1-\de^\nu\right)^2S_3+\la^{-2}\de^{2\nu}\sum_{j\neq3}S_j\\
&\leq \la^{-2}S_3\left((1-\de^\nu)^2+\de^{2\nu}\sum_{j\neq3}\frac{S_j}{S_3}\right).
\end{split}
\]
From Theorem \ref{thm:SetLambda} we know that for $j\neq 3$,
\[
\frac{S_j}{S_3}\lesssim e^{sN}
\]
Therefore, to bound these terms we use the definition of  $\de$ from Theorem \ref{thm:ToyModelOrbit} taking $\ga=\wt \ga (s-1)$.
Since $s-1>0$ is fixed, we can choose such $\wt \ga \gg 1$.
Then, we have that
\[
 \sum_{n\in\Lambda}|n|^{2s}\left|\beta^\la_n(0)\right|^2\lesssim \la^{-2}S_3.
\]
%From this statement and Theorem \ref{thm:NormalForm}, we can conclude that
%\[
% \left\|r(0)\right\|^2_{H^s}\lesssim \la^{-2}S_3.
%\]
To complete the proof of statement \eqref{eq:GrowthSobolev:Initial}, it is enough to 
recall that the support of $\Gamma-\mathrm{Id}$ is
\[
 \Lambda^3=\left\{n\in\ZZ^2: n=n_1-n_2+n_3, n_1,n_2,n_3\in\Lambda\right\}
\]
and apply Theorem \ref{thm:NormalForm}.

Using inequalities  \eqref{eq:GrowthSobolev:Final} and \eqref{eq:GrowthSobolev:Initial},
we have that
\[
\frac{\left\|r(T)\right\|^2_{H^s}}{\left\|r(0)\right\|^2_{H^s}}\gtrsim \frac{S_{N-1}}{S_3},
\]
and then, applying Theorem \ref{thm:SetLambda}, we obtain
\[
\frac{\|r(T)\|_{H^s}^2}{\|r(0)\|_{H^s}^2}\gtrsim \frac 12 2^{(s-1)(N-4)}
\ge \KK ^2.
\]
To obtain the mass estimate it is enough to use the following bound
 \[\left \|r(0)\right\|^2_{L^2}\lesssim \la^{-2}\,2^N.\]
To prove it, one can proceed as for $\|r(0)\|_{H^s}$ taking into account that $|\Lambda_j|=\sum_{n\in\Lambda_j}1=2^{N-1}$.

Finally, it only remains to estimate the diffusion time $T$. We use Theorem \ref{thm:SetLambda} to set $\KK\simeq 2^{(s-1)N/2}$ and
$c=4\kk \ga/(s-1)$, definition \eqref{def:LambdaOfN} to set
$\la=e^{\kk \ga N}\simeq \KK^{c/(2\ln2)}$.
%,formula (\ref{eq:GrowthSobolev:Initial}) to set $\mu=\la^{-1}2^N=
%\KK^{-c/(2\ln 2)}2^N\simeq \KK^{-c/(2\ln 2)+2/(s-1)}$.
Then, we obtain
\[
|T|\leq \KKK \ \ga\ \la^2 N^2  \leq \KKK\ \ga\ \KK^{c/\ln 2}
\dfrac{4\ln^2 \KK}{\ln^2 2\ (s-1)^2}\leq \KK^c
\]
for large $\KK$.
This completes the proof of Theorem \ref{thm:main}.
\end{proof}
\begin{proof}[Proof of Theorem \ref{thm:mainslow}]
It can be easily seen that Theorem \ref{thm:Approximation} is also valid taking as a new value for the parameter $\la$,
\[
\la=\mu\ii 2^{(N-1)/2}e^{\eta_2N^2/2}.
\]
With this choice, we are be able to control the intial Sobolev norm. Indeed, in the proof of Theorem \ref{thm:main} we have seen that 
\[
 \left \|r(0)\right\|^2_{H^s}\lesssim \la^{-2}S_3.
\]
Thus, to obtain a small initial Sobolev norm we need  $\la$ to satisfy $\la^{-2}S_3\leq \mu^2$. Using the estimates for $n\in\Lambda$ given in Theorem \ref{thm:SetLambda}, we have that
\[
 S_3\leq 2^{N-1}e^{\eta_2N^2}.
\]
Therefore, with this choice of $\la$, we have a initial Sobolev norm of size $\mu$. To prove the other estimates in Theorem \ref{thm:mainslow} it is enough to point out that now the growth of Sobolev norms is given by $\KK=\CCC/\mu$ and reproduce the proof of Theorem \ref{thm:main} with the different choice of $\la$.
\end{proof}

\section{The normal form: proof of Theorem \ref{thm:NormalForm}}\label{app:NormalForm}
To prove Theorem \ref{thm:NormalForm}, we consider as a change of variables
$\Gamma$  the time one map of a Hamiltonian vector field $X_F$, where
$F$ is the Hamiltonian
\[
 F=\frac{1}{4}\sum_{n_1,n_2,n_3,n_4\in\ZZ^2}F_{n_1n_2n_3n_4}\al_{n_1}\ol{\al_{n_2}},
\al_{n_3}\ol{\al_{n_4}}
\]
with
\begin{align*}
F_{n_1n_2n_3n_4}=&
\frac{-i}{\sum_{j=1}^4(-1)^{j+1}\left(|n_j|^2+v_{n_j}\right)}
\qquad\qquad\text{ if } (n_1,n_2,n_3,n_4)\in\II'\\
F_{n_1n_2n_3n_4}=&0 \qquad\qquad\qquad\qquad\qquad\qquad\qquad\quad\,\,\ \text{otherwise.}
\end{align*}
where $(n_1,n_2,n_3,n_4)\in\II'$ if either one of these two conditions are satisfied:
\begin{itemize}
 \item[(i)] $(n_1,n_2,n_3,n_4)\in (\ZZ^2)^4$ satisfies 
\[
n_1-n_2+n_3-n_4=0\quad \text{ and }\quad \left|\sum_{j=1}^4(-1)^{j+1}\left(|n_j|^2+v_j\right)\right|> \eta,
\] 
where $\eta$ is the constant given by Theorem \ref{thm:NormalForm} and at least one of the modes $n_j$  satisfies $n_j\in B(\kk_0)$, where $\kk_0$ is the constant defined in \eqref{def:kappa0}.
\item[(ii)] $(n_1,n_2,n_3,n_4)\in (\ZZ^2)^4$ satisfies 
\[
n_1-n_2+n_3-n_4=0\quad \text{ and }\quad |n_1|^2-|n_2|^2+|n_3|^2-|n_4|^2\neq 0
\] 
and each mode satisfies $n_j\in\ZZ^2\setminus B(\kk_0)$.
\end{itemize}

\begin{lemma}\label{lemma:WellDefinedNormalForm}
The Fourier coefficients  $F_{n_1n_2n_3n_4}$ are well defined and satisfy 
\[
 \left|F_{n_1n_2n_3n_4}\right|\leq 4.
\]
Thus, the function $F$ is well defined.
\end{lemma}
\begin{proof}
One only needs to proof that the denominators involved in the definition of the coefficients $F_{n_1n_2n_3n_4}$ do not vanish and to compute a lower bound for them. If $(n_1,n_2,n_3,n_4)\in \II'$ and satisfies condition $(i)$, the denominator do not vanish since its absolute value is bigger than $\eta$. Using the definition of $\eta$ in Theorem \ref{thm:NormalForm}, we obtain the wanted bound for $F_{n_1n_2n_3n_4}$.
If $(n_1,n_2,n_3,n_4)\in \II'$ and satisfies condition $(ii)$, it is enough to point out that $|n_1|^2-|n_2|^2+|n_3|^2-|n_4|^2$ is an integer number at therefore $|n_1|^2-|n_2|^2+|n_3|^2-|n_4|^2\neq 0$ implies
$\left||n_1|^2-|n_2|^2+|n_3|^2-|n_4|^2\right|\geq 1$. Then using that $|n_i|\geq \kk_0$ for $i=1,\ldots,4$ and \eqref{def:kappa0}, one can easily see that 
\[
 \left|\sum_{j=1}^4(-1)^{j+1}\left(|n_j|^2+v_{n_j}\right)\right|\geq \frac{1}{2}.
\]
\end{proof}

If we define $\Phi_F^t$ the flow of the vector field associated to the
Hamiltonian $F$, we have that
\[
 \begin{split}
\HH\circ\Gamma=&\left.H\circ   \Phi_F^t\right|_{t=1}\\
=&\HH+\{\HH,F\}+\int_0^1(1-t)\left\{\left\{\HH,F\right\},F\right\}
\circ\Phi_F^tdt\\
=&\DDD +\GG+\{\DDD, F\}\\
&+\{\GG,F\}+\int_0^1(1-t)\left\{\left\{\HH,F\right\},F\right\}\circ\Phi_F^tdt,
 \end{split}
\]
where $\{\cdot,\cdot\}$ denotes the Poisson bracket with respect to the symplectic form $\Omega=\frac{i}{2}\sum_{n\in\ZZ^2}\al_n\wedge \ol{\al_n}$. We define
\[
 \RRR=\{\GG,F\}+\int_0^1(1-t)\left\{\left\{\HH,F\right\},F\right\}\circ\Phi_F^tdt.
\]
Then, it only remains to obtain the desired bounds for $X_\RRR$ and $\Gamma$ and
to see that
\[
 \GG+\{\DDD, F\}=\wt \GG_1+\wt \GG_2.
\]
To obtain, this last equality, it is enough to use the definition for $F$ to see
that
\[
\begin{split}
  \GG+\{\DDD,
F\}=&\frac{1}{4}\sum_{n_1-n_2+n_3=n_4}\left(1-i\sum_{j=1}^4(-1)^{j+1}(|n_j|^2+v_{n_j})F_{n_1n_2n_3n_4}\right)\al_{
n_1}\ol{\al_{n_2}}\al_{n_3}\ol{\al_{n_4}}\\
=& \wt\GG_1+\wt\GG_2.
\end{split}
\]
To obtain the  bounds for $X_\RRR$ we proceed as in \cite{GuardiaK12}. We use the fact that the space $\ell^1$ is a Banach algebra with respect to the convolution
product. Namely, if $a,b\in \ell^1$ its convolution product $a\ast b$,
which is defined by
\[
( a\ast b)_n=\sum_{n_1+n_2=n} a_{n_1}b_{n_2}
\]
satisfies
\[
 \|a\ast b\|_{\ell^1}\leq \|a\|_{\ell^1}\| b\|_{\ell^1}.
\]
We start by bounding $X_{\{\GG,F\}}$, the
vector field associated to the Hamiltonian $\{\GG,F\}$. We have to bound
\[
\left\|X_{\{\GG,F\}}\right\|_{\ell^1}=2 \sum_{n\in \ZZ^2}
\left|\pa_{\ol{\al_n}}\{\GG,F\}\right|.
\]
Then,
\[
 \begin{split}
  \left\|X_{\{\GG,F\}}\right\|_{\ell^1}\leq& 2 \sum_{n,m\in \ZZ^2}
\left|\pa_{\ol{\al_n}}\left(\pa_{\ol{\al_m}}\GG\pa_{\al_m} F\right)\right|+ 2 \sum_{n,m\in
\ZZ^2} \left|\pa_{\ol{\al_n}}\left(\pa_{\al_m}\GG\pa_{\ol{\al_m}} F\right)\right|\\
\leq &2\sum_{n,m\in \ZZ^2} \left|\pa_{\ol{\al_n}\ol{\al_m}}\GG\right|\left|\pa_{\al_m}
F\right|+ 2 \sum_{n,m\in \ZZ^2} \left|\pa_{\ol{\al_m}}\GG\right|\left|\pa_{\ol{\al_n}\al_m}
F\right|\\
&+2\sum_{n,m\in \ZZ^2} \left|\pa_{\ol{\al_n}\al_m}\GG\right|\left|\pa_{\ol{\al_m}} F\right|+
2 \sum_{n,m\in \ZZ^2} \left|\pa_{\al_m}\GG\right|\left|\pa_{\ol{\al_n}\ol{\al_m}} F\right|.
 \end{split}
\]
All the terms can be bounded analogously. As an example, we bound the first one,
\[
\begin{split}
\sum_{n,m\in \ZZ^2} \left|\pa_{\ol{\al_n}\ol{\al_m}}\GG\right|\left|\pa_{\al_m}
F\right|\leq &16\sum_{n,m\in \ZZ^2} \left|\sum_{n_1+n_2=m+n}\al_{n_1}
{\al_{n_2}}\right|\left|\sum_{n_1-n_2+n_3=m}\ol{\al_{n_1}}\al_{n_2}\ol{\al_{n_3}}
\right|\\
\leq &16\sum_{n\in \ZZ^2} \sum_{n_1+n_2=n}|\al_{n_1}||\al_{n_2}|\sum_{m\in
\ZZ^2}\sum_{n_1-n_2+n_3=m}|\al_{n_1}||\al_{n_2}||\al_{n_3}|\\
\leq &\OO\left(\|\al\|_{\ell^1}^5\right),
\end{split}
\]
where, in the first line we have taken into account that
$|F_{n_1n_2n_3n_4}|\leq 4$, and  to obtain the last line we have used that
each sum in the previous line is a convolution product. The other term in the
reminder can be bounded analogously taking into account that
\[
\begin{split}
\int_0^1(1-t)\left\{\left\{\HH,F\right\},F\right\}
\circ\Phi_F^tdt=&\int_0^1(1-t)\left\{\wt\GG-\GG,F\right\}\circ\Phi_F^tdt\\
&+\int_0^1(1-t)\left\{\left\{\GG,F\right\},F\right\}\circ\Phi_F^tdt.
\end{split}
\]
Analogously, one can obtain bounds for $\Gamma-\mathrm{Id}$ recalling that
\[
 \Gamma=\Id+\int_0^1 X_F\circ \Phi_F^t dt.
\]

\section{Construction of the set $\Lambda$: proof of Theorem \ref{thm:SetLambda}}\label{app:SmallSobolev}
We devote this section to prove Theorem \ref{thm:SetLambda}. To this end we modify the construction of the set $\Lambda$ done in \cite{CollianderKSTT10} and \cite{GuardiaK12} for \eqref{def:NLS}. For equation \eqref{def:NLS} the resonant set was defined as 
\[
\begin{split}
 \BB=\Big\{(n_1,n_2,n_3,n_4)\in \left(\ZZ^2\right)^4:\,& n_1-n_2+n_3-n_4=0, \\&|n_1|^2-|n_2|^2+|n_3|^2-|n_4|^2=0, n_1,n_3\neq n_4\Big\}.
\end{split}
\]
In Section \ref{sec:SketchProof} we have defined the conditions $1_\La$ -- $6_\La$ which involve tuples in $\wt\AAA_1$. To avoid confusions, we denote by  $1'_\La$ -- $6'_\La$ the same conditions but referred to the tuples in $\BB$ (note that $\wt\AAA_1\subset\BB)$. In \cite{GuardiaK12} it is 
 proved the following proposition, which is a quantitative version of  Proposition 2.1 of \cite{CollianderKSTT10}. It constructs a set of modes $\Lambda^{(0)}$ satisfying conditions $1'_\La$ -- $6'_\La$.

\begin{proposition}[\cite{GuardiaK12}]\label{thm:SetLambda:GK}
Let $\KK\gg 1$. Then, there exists
$N\gg 1$ and a set $\Lambda^{(0)}\subset\ZZ^2$, with
\[
 \Lambda^{(0)}=\Lambda_1^{(0)}\cup\ldots\cup\Lambda_N^{(0)},
\]
which satisfies  conditions $1_\Lambda'$ -- $6_\Lambda'$ and also
\begin{equation}\label{def:Growth:GK}
 \frac{\sum_{n\in\Lambda^{(0)}_{N-1}}|n|^{2s}}{\sum_{n\in\Lambda^{(0)}_3}|n|^{2s}}\geq
 \dfrac 12 2^{(s-1)(N-4)}\ge \KK^2.
\end{equation}
Moreover, there exists $\eta_1'>0$ and $\eta_2'>0$ such that we can ensure that each generation
$\Lambda^{(0)}_j$ has $2^{N-1}$ disjoint frequencies $n$ satisfying 
\[
 e^{\eta_1' N^2}<|n|<e^{\eta_2' N^2}.
\]
\end{proposition}
This proposition is proved in Appendix C of \cite{GuardiaK12}. Moreover, following the proof, one can easily see that one can impose the following additional conditions.
\begin{itemize}
 \item[$9_\La'$] Consider any two modes $n_1,n_2\in\Lambda^{(0)}$. Then, they satisfy
\begin{equation}\label{cond:C}
|n_1+n_2|\neq |n_1-n_2|. 
\end{equation}
\item[$10_\La'$] Consider any two modes $n_1,n_2\in\Lambda^{(0)}$. Then, the lines passing through one of these points and perpendicular to the segment between them do not contain the origin. Equivalently, 
\begin{equation}\label{cond:C2}
(n_2-n_1)\cdot n_1\neq 0.
\end{equation}
\end{itemize}
These conditions will be used in the proof of Theorem \ref{thm:SetLambda}.

We use Proposition \ref{thm:SetLambda:GK} to prove Theorem \ref{thm:SetLambda}, which constructs a set $\Lambda$ satisfying properties $1_\La$ -- $8_\La$, which refer to  the resonant tuples  of the equation \eqref{def:NLSV}, that is to $\wt\AAA=\wt\AAA_0\cup\wt\AAA_1$. Note that if we take the set $\La^{(0)}$ given by Proposition \ref{thm:SetLambda:GK} and we define a new set
\[
 \Lambda'=\left\{n'\in\ZZ^2:n'=Cn, n\in\Lambda^{(0)}\right\},
\]
for some $0<C<e^{\nu N^2}$ with $\nu>0$, we obtain a new set which also satisfies all the conditions stated in Proposition \ref{thm:SetLambda:GK} for different constants $\eta_1'$ and $\eta_2'$. We will show that making the appropiate rescaling this new set $\Lambda'$ also satisfies conditions $1_\Lambda$ -- $8_\Lambda$.
%, which are referred to the resonant set $\wt\AAA=\wt\AAA_0\cup\wt\AAA_1$.

First, we point out that taking $N$ large enough, the points in the set $\Lambda^{(0)}$ satisfy 
\begin{equation}\label{def:LambdaRadius1}
|n|\geq\kk_0
%\,\, \text{ with }\,\, \kk_1=10\kk_0,
\end{equation}
where $\kk_0$ is the constant defined in \eqref{def:kappa0} (recall that $\kk_0$ is independent of $N$). Then, since the resonant tuples of $\wt\AAA_1$ are contained in the resonant tuples of $\BB$, one has that $\Lambda^{(0)}$ satisfies conditions $1_\La$ -- $6_\La$.

Now, we show that a suitable blow up of $\La^{(0)}$ satisfies  condition  $7_\La$ for the resonant set $\wt\AAA=\wt\AAA_0\cup\wt\AAA_1$. Condition $8_\La$ will need an additional blow up that will be done later on. To check condition $7_\La$, we consider three modes  $n_1,n_2,n_3\in\La^{(0)}$.
By Proposition \ref{thm:SetLambda:GK}, these three modes cannot make a rectangle with a mode out of $\Lambda^{(0)}$, in particular, they cannot make a rectangle with $n_4=0$. This implies
\[
 |n_1|^2-|n_2|^2+|n_3|^2\neq 0
\]
which is equivalent to 
\[
\left| |n_1|^2-|n_2|^2+|n_3|^2|\right|\geq 1
\]
since $|n_j|^2$ are natural numbers. 
%Define 
%\[
% \Theta_0=\max_{n\in B(\kk_0)\subset\ZZ^2}|v_n|,
%\]
%where $v_n$ are the Fourier coefficients of the potential $V$. The constant $\Theta_0$ is independent of $N$. 
We consider  the blow up of $\Lambda^{(0)}$,
\[
 \Lambda^{(1)}=\{n'\in\ZZ^2:n'=C_1n, n\in\Lambda^{(0)}\},
\]
with
\begin{equation}\label{def:C1}
 C_1=\left\lceil\kk_0^2+4\|V\|_{H^{s_0}}+1\right\rceil,
\end{equation}
where  $\lceil x\rceil$ denotes the upper integer part of $x$ and $\kk_0$ is the constant introduced in \eqref{def:LambdaRadius1}. Recall that $V\in H^{s_0}$ by hypothesis. We show that $\Lambda^{(1)}$ satisfies the condition $7_\La$. Note that the modes $n\in\Lambda^{(1)}$ satisfy
\begin{equation}\label{def:BoundLambda1}
 |n|\leq C_1 {e^{\eta_2'N^2}}.
\end{equation}
Take a mode $n_4'$ satisfying $|n_4'|< \kk_0$ and we show that they cannot form a resonant tuple in $\wt\AAA_0$.
%  Recall that by \eqref{def:LambdaRadius1}, the modes in $\La$ satisfy $|n_j|\geq \kk_0$, $j=1,2,3$ and thus the same happens for the corresponding 
%$n_j'$.
 Thanks to the blow up, we have that 
\[
\left| |n_1'|^2-|n_2'|^2+|n_3'|^2|\right|\geq \kk_0^2+4\|V\|_{H^{s_0}}+1
\]
We use this lower bound to prove that the tuple cannot be resonant. We need to check that $|\rr_{n_1'n_2'n_3'n_4'}|\geq \eta$ (see \eqref{def:SmallDivisor} for the definition of $\rr_{n_1'n_2'n_3'n_4'}$ and Theorem \ref{thm:NormalForm} for the definition of $\eta$). Using $|n_4'|< \kk_0$ and that the Fourier coefficients $v_n$ satisfy $|v_n|\leq \|V\|_{H^{s_0}}$, we have that
\[
\begin{split}
\left|\rr_{n_1'n_2'n_3'n_4'}\right|&\geq \left|\left| |n_1'|^2-|n_2'|^2+|n_3'|^2\right|-|n_4'|^2-\sum_{j=1}^4|v_{n_j'}|\right|\\
&\geq\left(\kk_0^2+4\|V\|_{H^{s_0}}+1\right)-\kk_0^2-4\|V\|_{H^{s_0}}= 1\geq \eta,
\end{split}
\]
and therefore the tuple cannot be resonant. This completes the proof of condition $7_\La$ for $\La^{(1)}$.

%This implies that all the resonant interactions belong to $\AAA_1$ and not $\AAA_0$. Since the resonant conditions in $\AAA_1$ coincide with the ones in $\BB$ and recalling that the conditions $2_\BB$ -- $5_\BB$ only involve the modes in $\La$, we have that $\La^{(1)}$ also satisfies $2_\AAA$ -- $5_\AAA$. 

Now it only remains to make an additional blow up so that the new set also satisfies $8_\La$. %We assume that, after such blow up, there can exist tuples in $\wt\AAA_0$ with two modes in $\Lambda$ and we will reach a contradiction. 
%In order to reach it, we will show that the modes out of $\Lambda$ belong to $\ZZ^2\setminus B(\kk_0)$ and therefore the tuple cannot belong to $\wt\AAA_0$. 
We define
%We consider the following blow up of $\Lambda^{(1)}$,
\[
 \Lambda^{(2)}=\{n'\in\ZZ^2:n'=C_2 n, n\in\Lambda^{(1)}\},
\]
with 
\begin{equation}\label{def:C2}
C_2= \left\lceil4(\kk_0+2\eta+8\|V\|_{H^{s_0}}+1)C_1e^{\eta_2'N^2}\right\rceil,
\end{equation}
where $C_1$ has been defined in \eqref{def:C1} and $\eta_2'$ is the constant introduced in Proposition \ref{thm:SetLambda:GK}. One can easily see that after the blow up, the set $\La^{(2)}$ still satisfies conditions $1_\La$ -- $7_\La$. Moreover, taking into account \eqref{def:BoundLambda1} and Proposition \ref{thm:SetLambda:GK}, we have that the points in $\Lambda^{(2)}$ satisfy
\begin{equation}\label{def:BoundLambda2}
 |n|\leq C_2C_1 {e^{\eta_2'N^2}}.
\end{equation}
We have to show that there cannot exist tuples in $\wt\AAA_0$ involving two modes $n_1',n_2'\in\Lambda$ and $m,\ell\not\in\Lambda$.
% Note that we can assume that the fourth $n_1$  satisfies $|n_1|<\kk_0$, since a resonant tuple in $\wt\AAA_0$ must involve a low harmonic and the others satisfy $|n|,|n_2'|,|n_3'| >\kk_0$. 
To belong to $\wt\AAA_0$, the tuples must have a mode in $B(\kk_0)$ and  must satisfy either
\[
|\rr_{m\ell n_1'n_2'}|\leq \eta\,\,\,\text{ or } |\rr_{mn_1'n_2'\ell}|\leq \eta
\]
(the other cases are symmetric). We show that this is not possible. We start with the first case. Namely, we assume that 
\[
 |\rr_{m\ell n_1'n_2'}|\leq \eta
\]
and we show that it implies $|m|>\kk_0$ and $|\ell|>\kk_0$.

Using that $m+\ell=n_1'+n_2'$ it is easy to see that the condition $|\rr_{m\ell n_1'n_2'}|\leq \eta$ can be written as
\[
 \left|\left|m-\frac{n_1'+n_2'}{2}\right|-\frac{|n_1'-n_2'|}{2}+v_m-v_\ell+v_{n_1'}-v_{n_2'}+\right|\leq \eta
\]
which, using that $|v_j|\leq \|V\|_{H^{s_0}}$, implies
\[
 \left|m-\frac{n_1'+n_2'}{2}\right|\leq \frac{|n_1'-n_2'|}{2}+\eta+4\|V\|_{H^{s_0}}.
\]
This condition is satisfied provided  $m$ is inside the circumference of radius $|n_1'-n_2'|/2+\eta+4\|V\|_{H^{s_0}}$ and center $(n_1'+n_2')/2$. We show now that  belonging to this circumference implies $|m|>\kk_0$. The point of the circumference closest to the origin is either one of these two points
\[
 z_\pm=\frac{n_1'+n_2'}{2}\left(1\pm \frac{|n_1'-n_2'|+2\eta+8\|V\|_{H^{s_0}}}{|n_1'+n_2'|}\right),
\]
which, satisfy
\[
\begin{split}
 |z_\pm|\geq&\frac{1}{2}\left||n_1'+n_2'|-\left|n_1'-n_2'\right|-2\eta-8\|V\|_{H^{s_0}}\right|\\
\geq&\frac{1}{2}\left|\frac{|n_1'+n_2'|^2-\left|n_1'-n_2'\right|^2}{|n_1'+n_2'|+\left|n_1'-n_2'\right|}-2\eta-8\|V\|_{H^{s_0}}\right|.
\end{split}
\]
Now recall that $n_1'=C_2 n_1$ and $n_2'=C_2 n_2$ where $C_2$ has been defined in \eqref{def:C2} and $n_1,n_2\in\Lambda^{(1)}$. Since $\Lambda^{(0)}$, and therefore also $\Lambda^{(1)}$, satisfy condition $9_\Lambda'$, we know that $|n_1+n_2|^2-\left|n_1-n_2\right|^2\neq 0$. Since it is an integer number, we can deduce that 
\[
 \left||n_1+n_2|^2-\left|n_1-n_2\right|^2\right|\geq 1,
\]
which implies
\[
 \left||n_1'+n_2'|^2-\left|n_1'-n_2'\right|^2\right|\geq C_2^2.
\]
On the other hand, using \eqref{def:BoundLambda2}, we have that 
\[
 |n_1'+n_2'|+\left|n_1'-n_2'\right|\leq 4C_2C_1 {e^{\eta_2'N^2}}.
\]
Therefore, the points $z_\pm$ satisfy
\[
  |z_\pm|\geq \frac{C_2}{4C_1 {e^{\eta_2'N^2}}}-2\eta-8\|V\|_{H^{s_0}}.
\]
Using the definition of $C_2$ in \eqref{def:C2} we have 
\[
  |z_\pm|\geq \kk_0+1.
\]
Since one of these points $z_\pm$ is the closest point of the circumference to the origin, we can deduce that  $|m|>\kk_0$. Note that the mode $\ell$ also must be inside the same circumference and therefore it also satisfies $|\ell|>\kk_0$.

Now we deal with the second case, that is 
\[
  |\rr_{mn_1'n_2'\ell}|\leq \eta.
\]
We assume that $m$ satisfies $|m|\leq \kk_0$ and we reach a contradiction. Using $m+n_2'=n_1+\ell$, we can write $\rr_{mn_1'n_2'\ell}$ as 
\[
\rr_{mn_1'n_2'\ell}=-2 (n_1'-n_2')\cdot n_1'+2 (n_1'-n_2')\cdot m +v_m-v_{n_1'}+v_{n_2'}-v_\ell.
\]
Therefore, 
\[
\left|\rr_{mn_1'n_2'\ell}\right|\geq2\left| (n_1'-n_2')\cdot n_1'\right|-2\left|(n_1'-n_2')\cdot m\right| -4\|V\|_{H^{s_0}}.
\]
For the first term, we use the fact that since $n_1',n_2'\in\Lambda^{(2)}$, they are of the form $n_j'=C_2n_j$ where $C_2$ is the constant given in \eqref{def:C2}. Using that $\La^{(0)}$, and also $\La^{(1)}$, satisfy condition $10_\La'$, we have that $\left| (n_1-n_2)\cdot n_1\right|\neq 0$. Since it is a natural number, it must satisfy $\left| (n_1-n_2)\cdot n_1\right|\geq 1$, which implies
\[
 \left| (n_1'-n_2')\cdot n_1'\right|\geq C_2^2.
\]
On the other hand, using \eqref{def:BoundLambda2} and that we have assumed that $|m|\leq \kk_0$, we have that 
\[
\left|(n_1'-n_2')\cdot m\right| \leq 2\kk_0C_2C_1 {e^{\eta_2'N^2}}.
\]
Therefore, 
\[
\left|\rr_{mn_1'n_2'\ell}\right|\geq 2 C_2^2-4\kk_0C_2C_1 {e^{\eta_2'N^2}}-\|V\|_{H^{s_0}}.
\]
Using the definition of $C_2$ in \eqref{def:C2}, one can easily see that
\[
 \left|\rr_{mn_1'n_2'\ell}\right|\geq1,
\]
which is a contradiction with what we had assumed. This implies that the set $\Lambda^{(2)}$ satisfies condition $8_\La$. Now, since we have considered blow ups by the constants $C_1$ and $C_2$, we just need to take any $\eta_1<\eta_1+\eta_2'$ and $\eta_2>2\eta_2'$ and increase $N$ if necessary, to obtain that for any $n\in\Lambda^{(2)}$,
\[
 e^{\eta_1N^2}\leq |n|\leq e^{\eta_2N^2}.
\]
This completes the proof of Theorem \ref{thm:SetLambda}.

\section{Proof of Approximation Theorem \ref{thm:Approximation}}\label{app:Approx}
We devote this section to proof the Approximation Theorem \ref{thm:Approximation}, which follows the scheme of the proof of Theorem 4 in \cite{GuardiaK12}. Nevertheless, note that now the remainder $\wt\RRR$ in \eqref{eq:VFFinal} has more terms than the corresponding one in \cite{GuardiaK12} due to the convolution potential. Throughout this section $C$ denotes any positive constant independent of $N$ and $\la$. 

We write equation \eqref{eq:VFFinal} as
\begin{equation}\label{eq:AfterNFInRotating}
- i \dot \bet_n = \EE_n(\bet)+ \wt\RRR_n(\bet),
\end{equation}
where $\EE:\ell^1\rightarrow \ell^1$ is the function defined as
\begin{equation}\label{def:Approx:CubicTerm}
 \EE_n(\bet)=-|\bet_n|^2\bet_n+  \sum_{(n_1,n_2,n_3)\in \AAA_0 (n)} \bet_{n_1} \overline { \bet_{n_2}}\bet_{n_3}e^{i\rr_{n_1n_2n_3n}t}+ \sum_{(n_1,n_2,n_3)\in \AAA_1 (n)} \bet_{n_1} \overline { \bet_{n_2}}\bet_{n_3}.
\end{equation}
Recall that we want to study the closeness of the orbit $\beta^\la(t)$ obtained in \eqref{def:RescaledApproxOrbit}, which is a solution of $-i\dot \bet^\la=\EE(\bet^\la)$, with the orbit $\wt\beta(t)$ of  equation \eqref{eq:AfterNFInRotating} which satisfies $\wt \bet(0)=\bet^\la(0)$. To this end,  we define the function $\xi$ as
\begin{equation}\label{def:Approx:Error}
 \xi=\wt\beta-\bet^\la,
\end{equation}
which satisfies $\xi(0)=0$. We apply refined Gronwall-like estimates to bound the $\ell^1$ norm of $\xi(t)$.

From equations
\eqref{eq:AfterNFInRotating} and \eqref{eq:VFFinal:Truncated} and recalling the definition of $\wt\RRR$ in \eqref{def:FinalRemainder}, one can obtain the  equation for $\xi$, which reads
\begin{equation}\label{eq:Approx:InfiniteODE}
\dot \xi =\ZZZ^0(t)+\ZZZ^1(t)\xi+\ZZZ^2(\xi,t)
\end{equation}
where
\begin{align}
 \ZZZ^0(t)=&\RRR'\left(\beta^\la\right)+\JJ\left(\beta^\la\right)\label{def:Approx:FirstIteration}\\
\ZZZ^1(t)=&D\EE\left(\beta^\la\right)+D\JJ\left(\beta^\la\right)\label{def:Approx:Linear}\\
\ZZZ^2(\xi,t)=&\EE\left(\beta^\la+\xi\right)-\EE\left(\beta^\la\right)-
D\EE\left(\beta^\la\right)\xi+
\RRR'\left(\beta^\la+\xi\right)-\RRR'\left(\beta^\la\right)\label{def:Approx:Higher}\\
&+\JJ\left(\beta^\la+\xi\right)-\JJ\left(\beta^\la\right)-
D\JJ\left(\beta^\la\right)\xi\notag
\end{align}
Applying the $\ell^1$ norm to this equation, we obtain
\begin{equation}\label{eq:Approx:Difference}
 \frac{d}{dt}\|\xi\|_{\ell^1}\leq \left\|\ZZZ^0(t)\right\|_{\ell^1}+\left\| \ZZZ^1(t)\xi\right\|_{\ell^1}+\left\| \ZZZ^2(\xi,t)\right\|_{\ell^1}.
\end{equation}
The next three lemmas give estimates for each term in the right hand side of this equation. 

%The proves of the first and third lemma are deferred to the end of the section. The proof of the second lemma follows the same lines as the proof of Lemma B.2 in \cite{GuardiaK12}.
\begin{lemma}\label{lemma:Approx:BoundFirstIteration}
The function $\ZZZ^0$
 defined in \eqref{def:Approx:FirstIteration}
 satisfies
$ \left\|\ZZZ^0\right\|_{\ell^1}\leq C\la^{-5} 2^{5N}$.
\end{lemma}

\begin{proof}
%[Proof of Lemma \ref{lemma:Approx:BoundFirstIteration}]
%Using the definition of $\ZZZ^0$ in \eqref{def:Approx:FirstIteration} and the definition of $\wt\RRR$ in \eqref{def:FinalRemainder}, we split $\ZZZ^0$ as $\ZZZ^0= \ZZZ^{01}+ \ZZZ^{02}$ with
%\[
%\begin{split}
% \ZZZ_n^{01}(t)&=\pa_{\ol\bet_n}\RRR'(\bet^\la)\\
% \ZZZ_n^{02}(t)&=\sum_{(n_1,n_2,n_3)\in\AAA_1(n)}\beta^\la_{n_1}\ol\beta^\la_{n_2}\beta^\la_{n_3}\left(e^{i\rr_{n_1n_2n_3n}t}-1\right).
%\end{split}
%\]
The bound of $\RRR'(\beta^\la)$ can be done as  in the proof of Lemma B.1 of \cite{GuardiaK12}. We bound now $\JJ(\beta^\la)$, which has been defined in \eqref{def:RemainderRotating}. The first step is to obtain upper bound for $\rr_{n_1n_2n_3n}$. Since the terms in $\JJ_n$ satisfy that $(n_1,n_2,n_3)\in\AAA_1(n)$ (see \eqref{eq:InfiniteODEAfterNormalForm}), we have that $|n_1|^2-|n_2|^2+|n_3|^2-|n|^2=0$. Then, recalling the definition of $\rr_{n_1n_2n_3n}$ in \eqref{def:SmallDivisor}, we have that 
\[
 \rr_{n_1n_2n_3n}=v_{n_1}-v_{n_2}+v_{n_3}-v_{n}.
\]
To bound it, we use the fact that $V\in H^{s_0}(\TT^2)$. Then,  its Fourier coefficients satisfy $\left|v_n\right|\leq C|n|^{-s_0}$. Since  by  Theorem \ref{thm:SetLambda} the modes in $\Lambda$ satisfy $|n|\geq e^{\eta_1 N^2}$, one has that 
\[
 \left|\rr_{n_1n_2n_3n}\right|\leq Ce^{-s_0\eta_1 N^2}.
\]
Using this estimate and that $\JJ$ is a convolution term, we have that
\[
 \left\|\JJ\left(\beta^\la\right)\right\|_{\ell^1}\leq Ce^{-s_0\eta_1 N^2} \left\|\beta^\la\right\|_{\ell^1}^3.
\]
To bound $\left\|\beta^\la\right\|_{\ell^1}$, it is enough to recall   that $\mathrm{supp}\{\beta^\la\}\subset\Lambda$ and the definition of $\beta^\la$ in \eqref{def:RescaledApproxOrbit} and Theorem \ref{thm:ToyModelOrbit} to show that
\[
 \left\|\beta^\la(t)\right\|_{\ell^1}\leq \sum_{n\in\Lambda}\left|\beta_n^\la(t)\right|\leq \la^{-1}2^N\sum_{j=1}^N \left|b_j\left(\la^{-2}t\right)\right|\leq C\la^{-1}2^NN.
\]
Thus, using the choice of $\la$ done in Theorem \ref{thm:Approximation},  we can conclude that
\[
 \left\|\JJ\right\|_{\ell^1}\leq C\la^{-3}e^{-s_0\eta_1 N^2}2^{3N}N^3\leq  C\la^{-5}2^{5N}.
\]
%Now it only remains to point out that from the results obtained in Theorem 3--bis,
%we know that at each time all but three components of $b$ are of size $|b_j|\lesssim \de^\nu$ for certain $\nu>0$ whereas the other two satisfy $|b_j|\leq 1$. Then, using the definition of $\de$ in Theorem \ref{thm:ToyModelOrbit}, we obtain that
%\[
%\sum_{j=1}^N \left|b_j\left(\la^{-2}t\right)\right|\leq C(1+N\de^\nu)\leq C,
%\]
%which implies
%\begin{equation}\label{eq:Beta:BoundL1}
% \left\|\beta^\la(t)\right\|_{\ell^1}\leq C2^N\la^{-1}.
%\end{equation}
%This  finishes the proof of the lemma.
\end{proof}

\begin{lemma}\label{lemma:Approx:BoundLinear}
The linear operator $\ZZZ^1(t)$
satisfies
$\left\| \ZZZ^1(t)\xi\right\|_{\ell^1}\leq \sum_{n\in\ZZ^2}f_n(t)|\xi_n|,$
where $f_n(t)$ are positive functions satisfying
\begin{equation}\label{eq:Approx:BoundLinearFs}
 \int_0^Tf_n(t)dt\leq C\ga N,
\end{equation}
where $T$ is the time given in \eqref{def:Time:Rescaled} and $\ga$ is the constant given in Theorem \ref{thm:ToyModelOrbit}.
\end{lemma}
The proof of this lemma follows the same lines as the proof of Lemma B.2 in \cite{GuardiaK12}.

To obtain estimates for  $\ZZZ^2(\xi,t)$ defined in  \eqref{def:Approx:Higher},
we apply bootstrap. 
Assume that for $0<t<T^*$ we have
\begin{equation}\label{cond:Bootstrap}
 \|\xi(t)\|_{\ell^1}\leq C\la^{-3/2}2^{-N}.
\end{equation}
\emph{A posteriori} we will show that the time \eqref{def:Time:Rescaled}
satisfies $0<T<T^*$ and therefore
the bootstrap assumption holds.
\begin{lemma}\label{lemma:Approx:BoundHigher}
Assume that condition \eqref{cond:Bootstrap} is satisfied. Then
the  operator $\ZZZ^2(\xi,t)$
satisfies
\[
 \left\| \ZZZ^2(\xi,t)\right\|_{\ell^1}\leq  C\la^{-5/2}\|\xi\|_{\ell^1}.
\]
\end{lemma}
The proof of this lemma follows the same lines as the proof of Lemma B.3 in \cite{GuardiaK12}.

%\begin{proof}
%[Proof of Lemma \ref{lemma:Approx:BoundHigher}]
%To prove Lemma \ref{lemma:Approx:BoundHigher}, we split $\ZZZ^2$ in \eqref{def:Approx:Higher} as $\ZZZ^2=\ZZZ^2_1+\ZZZ^2_2$ with
%\[
%\begin{split}
%\ZZZ^{21}(\xi,t)=&\EE\left(\beta^\la+\xi\right)-\EE\left(\beta^\la\right)-D\EE\left(\beta^\la\right)\xi\\
%\ZZZ^{22}(\xi,t)=&\wt\RRR\left(\beta^\la+\xi\right)-\wt\RRR\left(\beta^\la\right)\\
%\ZZZ^{23}(\xi,t)=&\sum_{(n_1,n_2,n_3)\in\AAA_1(n)}\left((\beta^\la_{n_1}+\xi_{n_1})(\ol\beta^\la_{n_2}+\ol\xi_{n_2})(\beta^\la_{n_3}+\xi_{n_3})-\beta^\la_{n_1}\ol\beta^\la_{n_2}\beta^\la_{n_3}\right)\left(e^{i\rr_{n_1n_2n_3n}t}-1\right).
%\end{split}
%\]
%The first two terms can be bounded as done in the proof of Lemma B.3 in \cite{GuardiaK12}. We bound here the third term
%\end{proof}

The rest of the proof of Theorem \ref{thm:Approximation} is analogous to the proof of Theorem 4 in \cite{GuardiaK12}. We reproduce it here. Applying Lemmas \ref{lemma:Approx:BoundFirstIteration},
\ref{lemma:Approx:BoundLinear} and \ref{lemma:Approx:BoundHigher} to the terms in equation \eqref{eq:Approx:Difference}, one obtains
\[ 
\sum_{n\in\ZZ^2} \frac{d}{dt}|\xi_n|=\frac{d}{dt}\|\xi\|_{\ell^1}\leq \sum_{n\in\ZZ^2}\left(f_n(t)+
 C\lambda^{-5/2}\right)|\xi_n|+C\la^{-5} 2^{5N}
\]
We apply a Gronwall-like argument to each harmonic of $\xi$. That is, we consider the following change of coordinates,
\begin{equation}\label{def:Approx:GronwallLikeChange}
 \xi_n=\zeta_n e^{\int_0^t\left(f_n(s)+C\lambda^{-5/2}\right)ds}
\end{equation}
to obtain
\[
\sum_{n\in\ZZ^2} e^{\int_0^t\left(f_n(s)+C\lambda^{-5/2}\right)ds} \frac{d}{dt}|\zeta_n|
\leq C\la^{-5} 2^{5N}.
\]
Taking into account that $f_n(t)+C\lambda^{-5/2}\geq 0$, we have that
\[
\frac{d}{dt}\|\zeta\|_{\ell^1}=\sum_{n\in\ZZ^2} \frac{d}{dt}|\zeta_n|\leq C\la^{-5} 2^{5N}.
\]
Integrating this equation, taking into account that $\zeta(0)=\xi(0)=0$  and using the bound for
$T$ in \eqref{def:Time:Rescaled} we obtain that
\[
 \|\zeta\|_{\ell^1}\leq C \la^{-3} 2^{5N}\ga N^2.
\]
To complete the proof we need to deduce from this estimate the bound for $\|\xi\|_{\ell^1}$. It only suffices to  use the change \eqref{def:Approx:GronwallLikeChange}, the estimate \eqref{eq:Approx:BoundLinearFs}
 and the definition of $T$ in \eqref{def:Time:Rescaled} to obtain
\[
 \|\xi\|_{\ell^1}\leq 2e^{C\ga N} \|\zeta\|_{\ell^1}\leq 2e^{C\ga N}\la^{-3} 2^{5N}\ga N^2.
\]
Therefore, using the condition on $\la$ from Theorem \ref{thm:Approximation} with any $\kk>C$ and taking $N$ big enough,
we obtain that for $t\in [0,T]$
\[
 \|\xi\|_{\ell^1}\leq \la^{-2}
\]
and therefore we can drop the bootstrap assumption \eqref{cond:Bootstrap}. This completes the proof of Theorem \ref{thm:Approximation}

\section*{Acknowledgements}
The author would like to thank Hakan Eliasson for proposing him to study the problem of growth of Sobolev norms for the nonliner Schr\"odinger equation with a convolution potential. He would also thank B. Gr\'ebert, D. Bambusi and L. Gauckler for pointing out some references. The author is  partially supported
by the Spanish MCyT/FEDER grants MTM2009-06973 and  MTM2012-31714 and the Catalan SGR grant 2009SGR859.

\bibliography{references}
\bibliographystyle{alpha}
\end{document}